\documentclass[11pt]{amsart}
\usepackage{mathrsfs}
\usepackage{threeparttable}
\usepackage{amsfonts}
\usepackage{amsmath}
\usepackage{amssymb}
\usepackage{enumerate}
\usepackage{latexsym}
\usepackage{xcolor}
\usepackage{enumitem}
\usepackage{enumerate}
\usepackage{makecell}
\usepackage{hyperref}
\usepackage[nameinlink,noabbrev]{cleveref}  
\usepackage{aliascnt}                       
\usepackage{listings}

\def\a{\alpha}    \def\b{\beta}      \def\l{\langle}   \def\r{\rangle}    
       \def\O{{\rm O}}   \def\Ga{\Gamma}   
                
\def\qed{\hfill $\Box$}                    
\def\pf{\noindent{\it Proof.}}

\newcommand\ZZ{\mathrm C}        \newcommand\D{\mathrm{D}}        
     
\newcommand\A{\mathrm{A}}        \newcommand\Sy{\mathrm{S}}       
\newcommand\Aut{\mathrm{Aut}}        \newcommand\Out{\mathrm{Out}}
\newcommand\Cay{\mathrm{Cay}}         
\newcommand\K{\mathsf{K}}        \newcommand\soc{\mathrm{soc}}    \newcommand\Core{\mathrm{Core}}

\newcommand\N{\mathbf{N}}                 
    
\newcommand\Om{\mathrm{\Omega}}

\newcommand\GL{\mathrm{GL}}    \newcommand\SL{\mathrm{SL}}           \newcommand\AGL{\mathrm{AGL}} 
\newcommand\PGL{\mathrm{PGL}}    \newcommand\PG{\mathrm{PG}} 
\newcommand\PSL{\mathrm{PSL}}  
\newcommand\SU{\mathrm{SU}}                
\newcommand\PSU{\mathrm{PSU}} 
               
\newcommand\Sp{\mathrm{Sp}}    \newcommand\PSp{\mathrm{PSp}}         
            
\newcommand\GammaL{\mathrm{\Gamma L}}           
\newcommand\AGammaL{\mathrm{A\Gamma L}}         
\newcommand\PGammaL{\mathrm{P\Gamma L}}           
\newcommand\POmega{\mathrm{P\Omega}}

\newcommand\E{\mathrm{E}}      \newcommand\F{\mathrm{F}}       \newcommand\G{\mathrm{G}}
       \newcommand\J{\mathrm{J}}
\newcommand\M{\mathrm{M}}      \newcommand\Q{\mathrm{Q}}       
\newcommand\HS{\mathrm{HS}}

\newtheorem{theorem}{Theorem}[section]
\newtheorem{hypothesis}[theorem]{Hypothesis}

\newtheorem{lemma}[theorem]{Lemma}
\newtheorem{definition}{Definition}[section]
\newtheorem{proposition}[theorem]{Proposition}
\theoremstyle{definition}
\newtheorem{example}[theorem]{Example}

\makeatletter
\def\subsection{\@startsection{subsection}{2}{\z@}%
   {1.5ex\@plus1ex \@minus.2ex}%
   {1.5ex \@plus.2ex}%
   {\normalfont\bfseries}}
\makeatother

\begin{document}
\title[2-arc-transitive Cayley graphs]{On 2-Arc-Transitive Cayley Graphs of Nonabelian Characteristically Simple Groups}
\thanks{Corresponding author: Hao Yu}
\thanks{2020 Mathematics Subject Classification. 05C25, 20B25.}
\thanks{The project was partially supported by NSF of Guangxi (2025GXNSFAA069013) and the NNSF of China (12571362, 12601653)}
\author[F. Deng, J.J. Li \and H. Yu]{%
Feng Deng, Jing Jian Li and Hao Yu}
\address{School of Mathematics \& Guangxi Base, Tianyuan Mathematical Center in Southwest China \& Guangxi Center for Mathematical Research \& Center for Applied Mathematics of Guangxi (GXU), 
  Guangxi University, Nanning, Guangxi 530004, P. R. China. }
\email{dengfeng@st.gxu.edu.cn (F. Deng);  lijjhx@gxu.edu.cn (J.J. Li);  haoyu@gxu.edu.cn (H. Yu).}

\begin{abstract}
The classification of symmetric Cayley graphs has attracted considerable attention. Results have been obtained for finite abelian groups and nonabelian simple groups. However, the problem remains open for finite nonabelian characteristically simple groups, although some results on normality for small valencies are known. 
In this paper, we determine the possible structures of automorphism groups $X$ of connected 2-arc-transitive Cayley graphs $\Ga$ of $G\cong T^k$, where $T$ is a finite nonabelian simple group and $k\geqslant2$. 
In the quasiprimitive case, we prove that $G\unlhd X$, $\Ga$ is complete, or $X$ is of product action type with restricted parameters. For the non-quasiprimitive case, we determine the possible structures of the quasiprimitive normal quotients of $\Ga$.

\vskip 5pt 

\noindent {\sc Keywords}. Cayley graph, 2-arc-transitive graph, characteristically simple group, quasiprimitive group

\end{abstract}

\maketitle

\section{Introduction}\label{sec1}

Throughout this paper, all graphs are finite, simple, undirected and connected. For a graph $\Ga$, let $V(\Ga)$, $E(\Ga)$ and $\Aut(\Ga)$ denote its vertex set, edge set and full automorphism group. 
For $X\leqslant \Aut(\Ga)$ and a positive integer $s$, an \textit{$s$-arc} in $\Ga$ is a sequence of vertices $(\a_0,\a_1,\ldots,\a_s)$ such that $\{\a_{i},\a_{i+1}\}\in E(\Ga)$ for $0\leqslant i\leqslant s-1$ and $\a_{j-1} \neq \a_{j+1}$ for $1\leqslant j\leqslant s-1$. 
The graph $\Ga$ is called \textit{$(X,s)$-arc-transitive} if $X$ is transitive on $V(\Ga)$ and on the $s$-arcs of $\Ga$. 
In particular, $\Ga$ is simply called \textit{$X$-arc-transitive} or \textit{$X$-symmetric} for $s=1$, and if $X=\Aut(\Ga)$, then $X$ is often omitted.

Let $\Ga$ be an $X$-symmetric graph with $N\unlhd X\leqslant \Aut(\Ga)$. The \textit{normal quotient graph} $\Ga_N$ of $\Ga$ relative to $N$ is defined as the graph with vertices the $N$-orbits in $V(\Ga)$, such that two $N$-orbits $U, W$ are adjacent (joined by an edge) in $\Ga_N$ if and only if at least one $x\in U$ and at least one $y\in W$ are adjacent in $\Ga$.
By the above definition, the valency of $\Ga_N$ necessarily divides the valency of the graph $\Ga$, and if the two are equal, then the graph $\Ga$ is called a \textit{normal cover of} $\Ga_N$.

Given a finite group $G$ and an inverse-closed subset $S$ of $G$ that does not contain the identity, that is, $1\not\in S\subset G$ and $S=S^{-1}:=\{s^{-1}\mid s\in S\}$, the \textit{Cayley graph} $\Cay(G,S)$ of $G$ is the graph with vertex set $G$ in which two vertices $x$ and $y$ are adjacent if and only if $yx^{-1}\in S$. 
It is easy to see that $\Cay(G,S)$ is connected exactly when $S$ generates $G$. 
The right regular representation of $G$ is contained in $\Aut(\Cay(G,S))$, and throughout the paper we identify $G$ with this regular subgroup. 
The graph $\Ga=\Cay(G,S)$ is said to be \textit{$X$-normal} if $G\unlhd X\leqslant \Aut(\Ga)$; in particular, if $X=\Aut(\Ga)$, then $\Ga$ is simply called \emph{normal}. Let $\Aut(G,S)=\{x\in\Aut(G) \mid S^x=S\}$. 
Since $1\in G$, the definition directly implies that $\Aut(G,S)\leqslant (\Aut(\Ga))_1$. It is well known that $\N_{\Aut(\Ga)}(G)=G\rtimes \Aut(G,S)$.

The study of symmetric graphs, which aims to characterize and construct such graphs, has attracted considerable attention in the literature. Tutte proved that no finite cubic graph is $s$-arc-transitive for $s\geqslant6$ \cite{T1947}, and Weiss extended this restriction by showing that no finite graph of valency at least three is 8-arc-transitive \cite{W1981}. For 2-arc-transitive Cayley graphs, classification results have been obtained for cyclic, dihedral and finite abelian groups \cite{ACMX1996,DMM2008,LP2008,M2003}. Li, Wang and Xia characterized $s$-arc-transitive Cayley graphs of finite nonabelian simple groups \cite{LWX2022}. For finite nonabelian characteristically simple groups, Hua and Feng studied cubic graphs admitting transitive characteristically simple groups \cite{HF2011}, while Du and Feng considered the pentavalent case \cite{DF2018}. However, these results mainly concern simple groups or graphs of small valencies, and the general structure of 2-arc-transitive Cayley graphs of finite nonabelian characteristically simple groups remains to be determined.

Quasiprimitive permutation groups and normal quotient graphs provide a framework for studying this problem. Praeger developed the theory of quasiprimitive groups and normal quotients \cite{BP2003,P1993,P2003}, while Daldegan and Schneider classified transitive nonabelian characteristically simple subgroups of finite quasiprimitive permutation groups \cite{DS2020}. Motivated by these results, we study connected $(X,2)$-arc-transitive Cayley graphs of $G\cong T^k$, where $T$ is a finite nonabelian simple group and $k\geqslant 2$.

\vskip 5pt 
We use the following hypothesis throughout the paper.

\begin{hypothesis}\label{hypothesis}
Let $\Ga=\Cay(G,S)$ be a connected $(X,2)$-arc-transitive Cayley graph of finite group $G=T_1\times\cdots\times T_k\cong T^k$, where each $T_i\cong T$ is finite nonabelian simple, $k$ is an integer with $k\geqslant 2$, and $G<X\leqslant \Aut(\Ga)$. When $X$ is not quasiprimitive on $V(\Ga)$, let $N$ be a maximal intransitive normal subgroup of $X$. Fix $\a\in V(\Ga)$, let $\overline{\a}=\a^N\in V(\Ga_N)$, and let $\overline{X}=X/N$ and $\overline{G}=GN/N$. 
\end{hypothesis}

The following theorem is the main result.

\begin{theorem}\label{main}
Let $\Ga$ be a graph described in Hypothesis \ref{hypothesis}. If $X$ is quasiprimitive, then $\Ga$ is $X$-normal, $\Ga$ is a complete graph, or $X$ is of product action type with $\soc(X)\cong Q^k$, where $(Q,T,k)$ is listed in Table \ref{tab:exact-pa}. Otherwise, $\overline{X}$ is quasiprimitive on $V(\Ga_N)$ and one of the following holds.
\begin{enumerate}[leftmargin=1.2em]
\item $\overline{X}$ is of affine type, $T\cong \PSL_2(7)$ and $\Ga_N\cong \K_{2^{3\ell}}$.
\item $\overline{X}$ is of almost simple type. Then either $\Ga_N$ is complete, or $\overline{G}\cong T$ and one of the following holds: $\overline{G}=\soc(\overline{X})$; $\Ga_N$ is $(\soc(\overline{X}),2)$-arc-transitive; or one of the possibilities in Tables \ref{tab:AS1}, \ref{tab:AS2}, \ref{tab:AS3}, \ref{tab:AS4} and \ref{tab:AS5} occurs.
\item $\overline{X}$ is of product action type with $\soc(\overline{X})\cong Q^r$ and $\overline{G}\cong T^\ell$, where $2\leqslant r$ and $1\leqslant\ell \leqslant k$. Then either $\overline{G}\unlhd \overline{X}$, or $\ell=r$ and $(Q,T,r)$ is listed in Table \ref{tab:equal-prime-r}. If $N\leqslant G$, then $(Q,T,r)$ also occurs in Table \ref{tab:exact-pa}.
\item $\overline{X}$ is of twisted wreath type and $\Ga_N$ is a Cayley graph of $\overline{G}=\soc(\overline{X})$.
\end{enumerate}
\end{theorem}

After this introductory section, Section \ref{sec2} introduces some notation, basic definitions and useful facts. Section \ref{sec3} gives the structural reduction for 2-arc-transitive Cayley graphs of finite characteristically simple groups. The quasiprimitive case is treated in Section \ref{sec4}, while the non-quasiprimitive case is considered in Section \ref{sec5}. The proof of Theorem \ref{main} is completed in Section \ref{sec6}.

\section{Preliminaries}\label{sec2}

The notation and terminology are standard and follow \cite{PS2018}.
For example, we use $T^m$ to denote the direct product of $m$ copies of $T$. For a group $X$ and a prime $r$, let $\O_r(X)$ denote the largest normal $r$-subgroup of $X$. Let $n$ be a positive integer. For a prime $p$, denote by $v_p(n)$ the largest nonnegative integer $a$ such that $p^a\mid n$, and denote by $\Pi(n)$ the prime divisors of $n$. Recall that $X^{(\infty)}$ is the smallest normal subgroup of a finite group $X$ such that $X/X^{(\infty)}$ is solvable.

\subsection{Groups}

We recall the following from \cite[Definition 2.2]{DS2020} regarding subgroups of direct products:
\begin{definition}\label{def1}
Let $M_1,\ldots,M_r$ be groups, let $N=M_1\times\cdots\times M_r$, and let $H\leqslant N$. Consider, for $i\in \{1, \ldots, r\}$, the projections $\pi_i: N\to M_i$ such that $(x_1, \ldots, x_r) \mapsto x_i$ for $x_i\in M_i$ and $1\leqslant i\leqslant r$. Then $H$ is called a {\it strip} of $N$ if $H\neq1$ and, for each $i$, either the restriction of $\pi_i$ to $H$ is injective or $\pi_i(H)=1$. 
Let $\mathcal{P}(H)=\{M_i \mid \pi_i(H)\neq 1\}$. Two strips $H$ and $K$ are said to be {\it disjoint} if $\mathcal{P}(H)\cap \mathcal{P}(K)=\emptyset$. The subgroup $H$ is a {\it subdirect subgroup} of $N$ if $\pi_i(H)=M_i$ for each $i$, and it is a {\it diagonal subgroup} of $N$ if the restriction of $\pi_i$ to $H$ is injective for each $i$. Further, $H$ is a {\it full diagonal subgroup} of $N$ if $H$ is both subdirect and diagonal. 
\end{definition}

A permutation group $X$ on $\Om$ is \textit{quasiprimitive} if every nontrivial normal subgroup of $X$ is transitive on $\Om$. We first recall a basic consequence of the classification of transitive characteristically simple subgroups of quasiprimitive groups.

\begin{proposition}\cite[Theorem 1.2]{DS2020}\label{subgroup}
Let $X$ be a finite quasiprimitive permutation group on $\Om$ with nonabelian socle, and let $G$ be a transitive nonabelian characteristically simple subgroup of $X$. Then $G\leqslant \soc(X)$.
\end{proposition}

We shall also use the following elementary description of the normal subgroups of a nonabelian characteristically simple group, see for example \cite[Theorem 4.16]{PS2018}. 

\begin{proposition}\label{obs} 
Let $G=T^k$ be a characteristically simple group, where $T$ is nonabelian simple and $k\geqslant 1$, and let $K\unlhd G$. Then, after relabeling if necessary, $K=T_1\times\cdots\times T_l\cong T^l$ for some $0\leqslant l\leqslant k$.
\end{proposition}

Let $q>1$ and $a>1$ be integers. A prime $z$ is called a \emph{primitive prime divisor} of $q^a-1$ if $z\mid q^a-1$ but $z\nmid q^i-1$ for every $1\leqslant i<a$.
 
\begin{proposition}[Zsigmondy]\cite[Theorem IX.8.3]{HB1982}\label{prop:ppd}
Let $q>1$ and $n>1$. Then $q^n-1$ has a primitive prime divisor except when $(q,n)=(2,6)$, or when $n=2$ and $q+1$ is a power of 2. In particular, if $z$ is a primitive prime divisor of $q^n-1$, then $n\mid z-1$.
\end{proposition}

\subsection{Graphs}

We first introduce several classes of graphs that will occur in the subsequent discussion.

\begin{example}\cite[Example 2.9]{H2024}
For integers $n\geqslant 2$, \textit{$n$-cube} $H(n,2)$ has vertex set $\{(x_1, x_2, \ldots, x_n)|x_i\in \{0, 1\}, i=1,\ldots,n\}$, and two vertices are adjacent if and only if they differ in exactly one coordinate. Simple observation shows that $H(n,2)$ is a Cayley graph of $H\cong \ZZ_2^n$ and valency $n$. 
\end{example}

\begin{example}\cite[Example 2.9]{H2024}\label{ex:folded-cube}
The \textit{folded $n$-cube} denoted by $\Box_n$, is the quotient of $H(n,2)$ by the partition into blocks of size 2 given by pairs of vertices that differ in all $n$ coordinates. Then $\Box_n$ is a 2-arc transitive Cayley graph of $\ZZ_2^{n-1}$ with valency $n$ and $\Aut(\Box_n)\cong \ZZ_2^{n-1}.\Sy_n$ whenever $n\geqslant 5$, $(\Sy_4\times\Sy_4)\rtimes\ZZ_2$ whenever $n=4$, and $\Sy_4$ whenever $n=3$. 
\end{example}

\begin{example}\cite[Example 5.1]{H2024}\label{ex:golay}
The graph $\Ga(C_{23})$ is a Cayley graph of $\ZZ_2^{11}$ with valency 23, diameter 3, and intersection array $\{b_0,b_1,b_2;c_1,c_2,c_3\}=\{23,22,21;1,2,3\}$, meaning that for any two vertices $x, y$ at distance $i$, exactly $c_i$ neighbors of $y$ lie at distance $i-1$ from $x$, and $b_i$ neighbors lie at distance $i+1$. Its full automorphism group is $\Aut(\Ga(C_{23}))\cong\ZZ_2^{11}.\M_{23}$.
The graph $\Ga(C_{22})$ is a Cayley graph of $\ZZ_2^{10}$ with valency 22, diameter 3, and intersection array
$\{22,21,20;1,2,6\}$. Moreover, $\Aut(\Ga(C_{22}))\cong\ZZ_2^{10}.(\M_{22}.2)$.
\end{example}

\begin{example}\cite[Section 1(1.3)]{IP1993}\label{ex:projective-space-graph}
Let $m\geqslant 3$, $a\geqslant 1$, $q=2^a$ and $n=(q^m-1)/(q-1)$. Let $\ZZ_2^n$ be the natural $\ZZ_2G_0$-permutation module, where $\PSL_m(q)\leqslant G_0\leqslant\PGammaL_m(q)$ acts on the $n$ points of $\PG_{m-1}(q)$. Let $W$ be the $\ZZ_2G_0$-submodule generated by the characteristic vectors of the complements of the hyperplanes of $\PG_{m-1}(q)$, and let $W^\perp$ be the orthogonal complement of $W$ in $\ZZ_2^n$ with respect to the standard bilinear form. Let $N=\ZZ_2^n/W^\perp$. The \emph{projective space graph} $P_m(a)$ has vertex set $N$, with $x+W^\perp$ adjacent to $y+W^\perp$ if and only if $x+y\in e_i+W^\perp$ for some standard basis vector $e_i$ of $\ZZ_2^n$. Thus $P_m(a)$ has valency $n$. By \cite[Theorem 5.3]{IP1993}, $N\cong \ZZ_2^{m^a}$ and $P_m(a)$ is a connected primitive affine $(X,2)$-arc-transitive graph, where $X=N\rtimes X_\a$, $\a=W^\perp$, $N$ is regular on $V(P_m(a))$ and $\PSL_m(q)\leqslant X_\a\leqslant \PGammaL_m(q)$. In particular, $P_m(1)\cong \K_{2^m}$.
\end{example}

The following property follows from \cite[Theorem 4.1]{P1993}.

\begin{proposition}\label{prop:NQ}
Let $\Ga$ be a connected $(X,2)$-arc-transitive graph, and let $N\unlhd X$ have at least three orbits on $V(\Ga)$. Then $N$ is semiregular on $V(\Ga)$, $X/N\leqslant \Aut(\Ga_N)$, $\Ga_N$ is $(X/N,2)$-arc-transitive, and $\Ga$ is an $N$-cover of $\Ga_N$. In particular, $\Ga$ and $\Ga_N$ have the same valency.
\end{proposition}

\section{A structural reduction}\label{sec3}

The following lemma gives a general description of the possibilities for the automorphism group of a connected Cayley graph of a finite characteristically simple group, based on \cite{DS2020}.

\begin{lemma}\label{lem:framework}
Let $G=T_1\times\cdots\times T_k\cong T^k$, where $T$ is finite nonabelian simple and $k\geqslant2$. Let $\Ga=\Cay(G,S)$ be connected, where $G\leqslant X\leqslant \Aut(\Ga)$. Then the following hold.
\begin{enumerate}[leftmargin=1.2em]
\item Suppose that $X$ is quasiprimitive on $V(\Ga)$ and that $\soc(X)=Q_1\times\cdots\times Q_r\cong Q^r$, where each $Q_i\cong Q$ is nonabelian simple. Then either $G\unlhd X$, or one of the following holds.
    \begin{enumerate}[leftmargin=1.5em]
    \item $X$ is of almost simple type and $\Ga$ is complete.
    \item $X$ is of simple diagonal type with unique minimal normal subgroup $\soc(X)\cong T^{k+1}$.
    \item $X$ is of product action type, and one of the following holds.
        \begin{enumerate}[leftmargin=1.5em]
        \item $\soc(X)\cong T^r$ for some $r>k\geqslant2$.
        \item $k=r$ and, after relabeling, $T_i<Q_i$ for every $i$.
        \item the subgroups $T_i$ are pairwise disjoint strips in $\soc(X)$, $|\mathcal{P}(T_i)|\in\{1,2\}$ for every $i$, and $|\mathcal{P}(T_j)|=2$ for some $j$. Moreover, $(T,Q)$ is listed in Table \ref{tab:groups}.
        \item $\soc(X)\cong(\A_n)^r$, $\soc(X)_\a\cong(\A_{n-1})^r$, where $|V(\Ga)|=n^r$ and $n\geqslant10$.
        \end{enumerate}
    \end{enumerate}
\item Suppose that $X$ is not quasiprimitive on $V(\Ga)$, and let $N\unlhd X$ be maximal intransitive on $V(\Ga)$. Let $\overline{X}=X/N$ and $\overline{G}=GN/N$. Then $\overline{X}$ is quasiprimitive on $V(\Ga_N)$ and $\overline{G}\cong T^\ell$ for some $1\leqslant\ell\leqslant k$. Let $\overline{G}=T_1\times\cdots\times T_\ell$ and $\soc(\overline{X})=P_1\times\cdots\times P_s$, where $P_i\cong P$ for some simple group $P$. Then one of the following five types occurs.
    \begin{enumerate}[leftmargin=1.5em]
    \item $\overline{X}$ is of affine type, $T\cong \PSL_2(7)$ and either $\Ga_N\cong \K_8$ for $\ell=1$, or $\Ga_N$ is a Cayley graph of $\ZZ_2^{3\ell}$ for $\ell\geqslant 2$.
    \item $\overline{X}$ is of almost simple type and $\overline{G}\leqslant \soc(\overline{X})$ is transitive on $V(\Ga_N)$. If $\ell\geqslant 2$, then $\Ga_N$ is complete.
    \item $\overline{X}$ is of simple diagonal type. Then one of the following holds.
        \begin{enumerate}[leftmargin=1.5em]
        \item $\ell=1$, $\soc(\overline{X})\cong T^2$, and $\overline{G}\cong T$ is regular on $V(\Ga_N)$. 
        \item $\ell\geqslant 2$, $\soc(\overline{X})\cong T^s$, $|V(\Ga_N)|=|T|^{s-1}$.
        \item $\ell=s=2$, $|V(\Ga_N)|=|P|$ and $(T,P)=(\A_5,\A_6)$, $(\M_{11},\M_{12})$ or $(\Omega_7(q),\POmega_8^+(q))$.
        \end{enumerate}
    \item $\overline{X}$ is of product action type. Then one of the following holds.
        \begin{enumerate}[leftmargin=1.5em]
        \item $\soc(\overline{X})\cong T^s$ for some $s\geqslant\ell\geqslant 2$.
        \item $\ell=1$, $s=2$, and $(T,P)$ is listed in Table \ref{tab:groups}.
        \item $\ell=s\geqslant 2$ and, after relabeling, $T_i<P_i$ for every $i$.
        \item $\ell\geqslant 2$, the subgroups $T_i$ are pairwise disjoint strips in $\soc(\overline{X})$, $|\mathcal{P}(T_i)|\in\{1,2\}$ for every $i$, and $|\mathcal{P}(T_j)|=2$ for some $j$. In addition, $(T,P)$ is listed in Table \ref{tab:groups}.
        \item $\ell\geqslant 2$ and $\soc(\overline{X})\cong(\A_t)^s$, where $|V(\Ga_N)|=t^s$ and $t\geqslant 10$.
        \end{enumerate}
    \item $\overline{X}$ is of twisted wreath type, and $\Ga_N$ is a Cayley graph of $\overline{G}=\soc(\overline{X})$.
    \end{enumerate}
\end{enumerate}
\end{lemma}

\begin{table}[htbp]
\centering
\renewcommand{\arraystretch}{1.25}
\setlength{\tabcolsep}{7pt}
\caption{Pairs $(T,Q)$ in the strip case.}
\label{tab:groups}
{\scriptsize\begin{tabular}{|c|c|c|c|c|c|}
\hline
$T$ & $Q$ & Condition          & $T$ & $Q$ & Condition \\
\hline
$\A_6$ & $\A_6$ &              & $\POmega_8^+(q)$ & $\A_n$ & $n=|\POmega_8^+(q):\Omega_7(q)|$ \\
\hline
$\M_{12}$ & $\M_{12}$ &             & $\POmega_8^+(2)$ & $\Sp_8(2)$ & \\
\hline
$\M_{12}$ & $\A_{12}$ &             & $\Sp_4(2^a)$ & $\Sp_{4b}(2^{a/b})$ & $a\geqslant 2$, $b\mid a$ \\
\hline
$\POmega_8^+(q)$ & $\POmega_8^+(q)$ & $q\geqslant 2$            & $\Sp_4(2^a)$ & $\A_n$ & $a\geqslant 2$, $n=|\Sp_4(2^a):\Sp_2(2^{2a}).\ZZ_2|$ \\
\hline
\end{tabular}}
\end{table}

\pf\, Suppose first that $X$ is quasiprimitive. If $G\unlhd X$, then the conclusion follows from the normality of $G$ in $X$. We may therefore assume that $G\ntrianglelefteq X$. Since $|V(\Ga)|=|T|^k$ is not a prime power, $X$ is not of affine type. Thus $\soc(X)$ is nonabelian, and Proposition \ref{subgroup} gives $G\leqslant\soc(X)$. Let $\soc(X)=Q_1\times\cdots\times Q_r\cong Q^r$, where each $Q_i\cong Q$ is finite simple. The equality $G=\soc(X)$ would imply $G\unlhd X$, contrary to the assumption. Hence $G<\soc(X)$.

Assume that $G$ contains a minimal normal subgroup $Q_i$ of $\soc(X)$. Since $Q_i\unlhd\soc(X)$ and $G\leqslant \soc(X)$, we have $Q_i\unlhd G$. Proposition \ref{obs} implies that $Q_i$ is one of the simple direct factors of $G$; in particular, $Q\cong T$ and $\soc(X)\cong T^r$. Since $G<\soc(X)$, we get that  $r>k$. The almost simple type cannot occur because $r\geqslant 2$. The twisted wreath type cannot occur because its socle is regular on $V(\Ga)$. The transitive subgroup $G\leqslant \soc(X)$ would then satisfy $G=\soc(X)$, contrary to $G\ntrianglelefteq X$. If $X$ is of simple diagonal type, then $\soc(X)_\a\cong T$ and $|T|^k=|V(\Ga)|=|\soc(X):\soc(X)_\a|=|T|^{r-1}$. Thus $r=k+1$, which gives part (1)(b). The product action type gives part (1)(c)(i).

Next assume that $G$ contains no minimal normal subgroup of $\soc(X)$. By \cite[Corollary 1.4]{DS2020}, the group $X$ is either of almost simple type or of product action type arising from a group of almost simple type. In the former case, \cite[Theorem 1.3 (1)]{DS2020} shows that $\soc(X)=\A_n$ in its natural action of degree $n\geqslant 10$; and so $\Ga\cong \K_n$. In the latter case, \cite[Theorem 1.3 (3)]{DS2020} gives parts (1)(c)(ii)--(iv). This proves part (1).

Suppose now that $X$ is not quasiprimitive. The maximal intransitive normal subgroup $N$ has at least three orbits; otherwise, two orbits would induce an index-2 subgroup of $G$, contrary to $G'=G$. 
Let $\mathcal{B}$ be the set of $N$-orbits on $V(\Ga)$, and let $K$ be the kernel of the action of $X$ on $\mathcal{B}$. Then $N\leqslant K$ and $K$ is intransitive. The maximality of $N$ gives $K=N$. Thus $X/N$ acts faithfully on $V(\Ga_N)$, and $X/N\leqslant\Aut(\Ga_N)$.
If $M/N$ is a nontrivial normal subgroup of $X/N$ that is intransitive on $\mathcal{B}$, then $M$ is an intransitive normal subgroup of $X$ properly containing $N$, a contradiction. Thus $\overline{X}=X/N$ is quasiprimitive on $V(\Ga_N)$. 
Let $\soc(\overline{X})=P_1\times\cdots\times P_s\cong P^s$, where each $P_i\cong P$ is finite simple.
The subgroup $G\cap N$ is normal in $G$, and $G\nleqslant N$ because $G$ is transitive whereas $N$ is intransitive. Proposition \ref{obs} and the second isomorphism theorem give $\overline{G}=GN/N\cong G/(G\cap N)\cong T^\ell$ for some $1\leqslant\ell\leqslant k$. Note that $\overline{G}$ is transitive on $V(\Ga_N)$.

Assume first that $\ell=1$. Then $\overline{G}\cong T$ is a transitive nonabelian simple subgroup of $\overline{X}$. 
If $\overline{X}$ is primitive, \cite[Proposition 6.1]{P1990} and \cite[Theorem 1.2]{BP2003} give the affine, almost simple, simple diagonal and product action types relevant here. 
For the affine type, \cite[Proposition 5.2]{P1990} gives $T\cong \PSL_2(7)$, $\overline{X}\cong \AGL_3(2)$ and $|V(\Ga_N)|=8$, so $\Ga_N\cong \K_8$. 
Under the almost simple type, \cite[Proposition 6.2]{P1990} gives $\overline{G}\leqslant \soc(\overline{X})$. 
The simple diagonal type is described in \cite[Section 4.3]{BP2003}, where $\soc(\overline{X})\cong T^2$ and $\overline{G}\cong T$ is regular on $V(\Ga_N)$, giving part (2)(c)(i). 
For product action, \cite[Section 4.5]{BP2003} gives $\soc(\overline{X})\cong P^2$ and one of the pairs $(T,P)$ listed in Table \ref{tab:groups}. 
If $\overline{X}$ is quasiprimitive but imprimitive, \cite[Theorem 2]{P2003} gives either the almost simple type, with $\overline{G}$ transitive in $\soc(\overline{X})$, or the product action type, with $\soc(\overline{X})\cong P^2$, $T\cong \Sp_4(2^a)$ and $P\cong \Sp_{4b}(2^{a/b})$ for the parameters in Table \ref{tab:groups}. Thus part (2) holds when $\ell=1$.

Assume that $\ell\geqslant 2$. If $\soc(\overline{X})$ is abelian, then $\overline{X}$ is of affine type. By \cite[Theorem 1.3 (0)]{DS2020}, $T\cong \SL_3(2)\cong \PSL_2(7)$ and $|V(\Ga_N)|=8^\ell$. The regularity of $\soc(\overline{X})$ gives $\soc(\overline{X})\cong \ZZ_2^{3\ell}$, yielding part (2)(a).

Suppose now that $\soc(\overline{X})$ is nonabelian. Proposition \ref{subgroup} gives $\overline{G}\leqslant \soc(\overline{X})$. 
If $\overline{G}$ contains a minimal normal subgroup of $\soc(\overline{X})$, then Proposition \ref{obs} gives $P\cong T$ and so $\soc(\overline{X})\cong T^s$. The almost simple type is excluded by $\ell\geqslant 2$. For the simple diagonal type, a point stabilizer in $\soc(\overline{X})$ is isomorphic to $T$, and so $|V(\Ga_N)|=|T|^{s-1}$, giving part (2)(c)(ii). 
The product action type has $\overline{G}\cong T^\ell\leqslant T^s$, whence $\ell\leqslant s$, and yields part (2)(d)(i). 
In the twisted wreath case, $\soc(\overline{X})$ is regular on $V(\Ga_N)$, while $\overline{G}\leqslant \soc(\overline{X})$ is transitive. Hence $|\overline{G}|\geqslant|V(\Ga_N)|=|\soc(\overline{X})|$, and the inclusion forces $\overline{G}=\soc(\overline{X})$, giving part (2)(e).

It remains to consider the case where $\overline{G}$ contains no minimal normal subgroup of $\soc(\overline{X})$. By \cite[Theorem 1.3]{DS2020}, the remaining possibilities occur in the almost simple, simple diagonal and product action types. The almost simple possibility has $\overline{X}=\A_n$ or $\Sy_n$ in its natural action, so $\Ga_N$ is complete. For the simple diagonal type, $\ell=s=2$, $T_i<P_i$, and $(T,P)$ is one of $(\A_5,\A_6)$, $(\M_{11},\M_{12})$ or $(\Omega_7(q),\POmega_8^+(q))$ with $q\geqslant2$, giving part (2)(c)(iii). The product action possibilities give parts (2)(d)(iii)--(v), while no additional twisted wreath case occurs.  \qed

\section{The quasiprimitive case}\label{sec4}

In this section, we always let $\Ga=\Cay(G,S)$ be a connected $(X,2)$-arc-transitive Cayley graph of finite group $G=T_1\times\cdots\times T_k\cong T^k$, where each $T_i\cong T$ is finite nonabelian simple, $k$ is an integer with $k\geqslant 2$, and $G<X\leqslant \Aut(\Ga)$. Suppose that $X$ is quasiprimitive on $V(\Ga)$. Then $X$ is of one of the following five types: affine, almost simple, simple diagonal, product action or twisted wreath (see \cite[Theorem 1]{P1993} for more details). The quasiprimitive case under Hypothesis \ref{hypothesis} is given by the following theorem.

\begin{theorem}\label{thm:qprim}
Let $\Ga$ be as in Hypothesis \ref{hypothesis}. Suppose that $X$ is quasiprimitive on $V(\Ga)$. Then $\Ga$ is $X$-normal, $\Ga$ is complete, or $X$ is of product action type with $\soc(X)\cong Q^k$ and $(Q,T,k)$ listed in Table \ref{tab:exact-pa}.
\end{theorem}

The five types are considered separately.

\begin{lemma}\label{lem:affine1}
The group $X$ is not of affine type.
\end{lemma}

\pf\, Suppose that $X$ is of affine type. Then $\soc(X)\cong \ZZ_p^d$ for some prime $p$ and positive integer $d$, and $\soc(X)$ is regular on $V(\Ga)$. 
Both $\soc(X)$ and $G$ are regular on $V(\Ga)$, so $|G|=|V(\Ga)|=p^d$. This is impossible because $G\cong T^k$ with $T$ nonabelian simple. \qed

\begin{lemma}\label{lem:AS1}
Suppose that $X$ is of almost simple type. Then either $\Ga$ is $X$-normal or $\Ga$ is complete.
\end{lemma}

\pf\ If $G\unlhd X$, then $\Ga$ is $X$-normal. Otherwise, $G\ntrianglelefteq X$, and Lemma \ref{lem:framework} (1)(a) yields that $\Ga$ is a complete graph. \qed

\vspace{0.5em}

By \cite[Theorem 2]{P1993}, if $\Ga$ is connected $(X,2)$-arc-transitive, and $X$ is quasiprimitive on $V(\Ga)$, then $X$ is of affine, almost simple, product action or twisted wreath type. This gives the following lemma.

\begin{lemma}\label{lem:diagonal1}
The group $X$ is not of simple diagonal type.
\end{lemma}

\vspace{0.5em}

Before proving that $X$ is of product action type, we establish two lemmas that will be needed in this setting. For a nontrivial finite group $Y$, denote $\mu(Y)=\min\{|Y:M| : M<Y\}$. 

\begin{lemma}\label{lem:equal-prime-divisors}
Let $Q$ and $T$ be finite nonabelian simple groups with $T<Q$ and $\Pi(|T|)=\Pi(|Q|)$. Suppose that there are an integer $r\geqslant 2$ and a nontrivial subgroup $R\leqslant Q^r$ such that $R\cong R_0$ for some proper subgroup $R_0<Q$, and $|Q|^r=|T|^r|R|/|T^r\cap R|$. Then $(Q,T,r)$ is listed in Table \ref{tab:equal-prime-r}. Moreover, if $T^r\cap R=1$, then row 5 does not occur.
\end{lemma}

\begin{table}[htbp]
\centering
\renewcommand{\arraystretch}{1.25}
\setlength{\tabcolsep}{4.2pt}
\caption{Possible triples $(Q,T,r)$ in Lemma \ref{lem:equal-prime-divisors}.}
\label{tab:equal-prime-r}
{\scriptsize
\begin{tabular}{|c|c|c|c|p{5.4cm}|}
\hline
Row & $Q$   & $T$ & $r$ & Conditions\\
\hline
1   & $\A_c$               & $\A_h$             & $r\geqslant 2$ & $5\leqslant h<c$, $(h,c]$ contains no prime, $d=c!/h!$, $rv_p(d)\leqslant v_p(|Q|)$ for $p\mid d$, and $d^r c\leqslant c!/2$\\
\hline
2   & $\PSp_{2m}(q)$      & $\POmega^-_{2m}(q)$ & 2 & $m\geqslant 4$ is even and $q$ is even\\
\hline
3   & $\POmega_{2m+1}(q)$ & $\POmega^-_{2m}(q)$ & 2 & $m\geqslant 4$ is even and $q$ is odd\\
\hline
4   & $\POmega^+_{2m}(q)$ & $\POmega_{2m-1}(q)$ & 2 & $m\geqslant 4$ is even\\
\hline
5   & $\Sp_6(2)$          & $\A_8$              & 2 & impossible if $T^2\cap R=1$\\
\hline
6   & $\M_{11}$           & $\PSL_2(11)$        & 2 & \\
\hline
7 & $\M_{12}$         & $\M_{11}$ & 2 & \\
\hline
8 & $\M_{24}$         & $\M_{23}$ & 2, 3 & \\
\hline
\end{tabular}}
\end{table}

\pf\, Let $d=|Q:T|$. The order identity $d^r=|R:T^r\cap R|$ implies that $d^r\mid |R_0|$. Let $M$ be a maximal subgroup of $Q$ containing $R_0$. Since $d^r\mid |R_0|$ and $R_0\leqslant M$, we have $d^r\mid |M|$. Hence $\mu(Q)\leqslant |Q:M|\leqslant |Q|/d^r$, giving $d^r\mu(Q)\leqslant |Q|$. 
Moreover, for each prime $p\mid d$, the divisibility $d^r\mid |R_0|$ yields $rv_p(d)\leqslant v_p(|R_0|)\leqslant v_p(|Q|)$. It follows that $r\leqslant\min_{p\mid d}\lfloor v_p(|Q|)/v_p(d)\rfloor$.

Since $T<Q$ and $\Pi(|T|)=\Pi(|Q|)$, \cite[Corollary 5 and Table 10.7]{LPS2000} gives either $(Q,T)=(\A_c,\A_h)$ with $5\leqslant h<c$ and no prime $p$ satisfying $h<p\leqslant c$, or $(Q,T)$ is one of $(\PSp_{2m}(q),\POmega^-_{2m}(q))$ with $m$ and $q$ even, $(\POmega_{2m+1}(q),\POmega^-_{2m}(q))$ with $m$ even and $q$ odd, $(\POmega^+_{2m}(q),\POmega_{2m-1}(q))$ with $m\geqslant 4$ even, $(\PSp_4(q),\PSL_2(q^2))$ with $q>2$, or one of the specific pairs listed in \cite[Table 10.7]{LPS2000}. The case $Q=\PSL_2(8)$ is excluded because every proper subgroup of $Q$ is solvable; see \cite[Proposition 4.1.17]{KL1990}.

For $(Q,T)=(\A_c,\A_h)$, the equalities $d=c!/h!$ and $\mu(\A_c)=c$ give row 1 of Table \ref{tab:equal-prime-r}. For the three infinite classical families with $m\geqslant 4$, Proposition \ref{prop:ppd} supplies a primitive prime divisor $z$ of $q^m-1$, except when $(q,m)=(2,6)$. In the former cases, $v_z(|Q|)=2v_z(d)$, so $rv_z(d)\leqslant 2v_z(d)$ and so $r=2$. For $(q,m)=(2,6)$, the equalities $v_7(|Q|)=2$ and $v_7(d)=1$ again give $r=2$. 
It remains to consider $(Q,T)=(\PSp_4(q),\PSL_2(q^2))$. Writing $\delta=(2,q-1)$, we have $d=q^2(q^2-1)$ and $|Q|/d^2=(q^2+1)/\delta$. By \cite[Theorem 5.2.2 and Table 5.2.A]{KL1990}, $\mu(\PSp_4(q))=(q^4-1)/(q-1)$ for $q>3$, and $\mu(\PSp_4(3))=27$. Thus $\mu(Q)>|Q|/d^2$, contradicting $d^r\mu(Q)\leqslant |Q|$ for $r\geqslant 2$.

For the remaining pairs in \cite[Table 10.7]{LPS2000}, the bounds $rv_p(d)\leqslant v_p(|Q|)$ and $d^r\mu(Q)\leqslant |Q|$, together with \cite{Atlas}, leave $(\PSL_6(2),\PSL_5(2),2)$, $(\Sp_6(2),\A_8,2)$, $(\POmega^+_8(2),\A_9,2)$ and the sporadic cases in rows 6--8 of Table \ref{tab:equal-prime-r}, with $r=3$ also possible for $(\M_{12},\M_{11})$. 
For $(Q,T)=(\PSL_6(2),\PSL_5(2))$ and $(\POmega^+_8(2),\A_9)$, the values of $|Q|/d^2$ are 4960 and 189, respectively. The divisibility $d^2\mid |R_0|$ requires $|Q:M|\mid |Q|/d^2$ for every maximal subgroup $M$ containing $R_0$, but the maximal-subgroup indices in \cite{Atlas} exclude both cases. Similarly, $(Q,T,r)=(\M_{12},\M_{11},3)$ is excluded because $|Q|/d^3=55$ and no maximal subgroup of $\M_{12}$ has index dividing 55. Hence rows 5--8 of Table \ref{tab:equal-prime-r} remain.

Suppose that $T^r\cap R=1$ and row 5 holds. Then $(Q,T,r)=(\Sp_6(2),\A_8,2)$, $d=72$ and $|R|=d^2=5184$. Since $d^2\mid |R_0|\leqslant |R|$, we have $|R_0|=5184$ and $|Q:R_0|=280$. Let $R_0\leqslant M<Q$ with $M$ maximal. By \cite[Page 46]{Atlas}, $|Q:M|=28$ and $M\cong \PSU_4(2):\ZZ_2$, so $|M:R_0|=10$. Since $|\soc(M)|=25920> |R_0|$, the simple group $\soc(M)\cong \PSU_4(2)$ acts nontrivially on the ten cosets of $R_0$ in $M$. This contradicts $\mu(\PSU_4(2))=27$. Thus row 5 is excluded when $T^r\cap R=1$. \qed

\begin{lemma}\label{lem:strip-pair-exclusion}
Let $(Q,T)$ be a pair in Table \ref{tab:groups} with $Q\not\cong T$. Then the following hold.
\begin{enumerate}[label=(\roman*),leftmargin=2em]
\item $\Pi(|Q|)\neq \Pi(|T|)$.
\item There is no solvable proper subgroup $R<Q$ such that $Q=TR$.
\end{enumerate}
\end{lemma}

\pf\, We first prove (i). The pair $(Q,T)=(\A_{12},\M_{12})$ cannot occur because $7\mid |Q|$ but $7\nmid |T|$. For $(Q,T)=(\Sp_8(2),\POmega_8^+(2))$, the prime 17 divides $|Q|$ but not $|T|$, which rules out this case as well.
Let $Q=\A_n$ and $T=\POmega_8^+(q)$, where $n=|\POmega_8^+(q):\Omega_7(q)|=\frac{(2,q-1)}{(4,q^4-1)}q^3(q^4-1)$. Every prime divisor of $|T|$ is at most $q^6-1$. For $q=2$ or 3, the prime $11$ divides $|Q|$ but not $|T|$. Assume that $q\geqslant 4$ is even or $q\geqslant 5$ is odd. If $q$ is even, then $n=q^3(q^4-1)$ and $n-2(q^6-1)=q^6(q-2)-q^3+2>0$. If $q$ is odd, then $n=q^3(q^4-1)/2$ and $n-2(q^6-1)=(q^6(q-4)-q^3+4)/2>0$. Thus $2(q^6-1)<n$. By \cite[Theorem 418]{HW1960}, there is a prime $z$ such that $q^6-1<z<2(q^6-1)<n$. Since $z>2$ and $|Q|=n!/2$, we have $z\mid |Q|$, whereas $z>q^6-1$ gives $z\nmid |T|$.

Consider $Q=\A_n$ and $T=\Sp_4(q)$, where $q=2^a\geqslant 4$ and $n=q^2(q^2-1)/2$. The order $|T|=q^4(q^2-1)^2(q^2+1)$ shows that every prime divisor of $|T|$ is at most $q^2+1$. A direct calculation gives $n-2(q^2+1)=(q^4-5q^2-4)/2>0$, so $2(q^2+1)<n$. By \cite[Theorem 418]{HW1960}, there is a prime $z$ satisfying $q^2+1<z<2(q^2+1)<n$. Now $z>2$ and $|Q|=n!/2$, which implies $z\mid |Q|$, whereas the bound $z>q^2+1$ gives $z\nmid |T|$.

It remains to consider $Q=\Sp_{4b}(q_0)$ and $T=\Sp_4(q_0^b)$, where $q_0=2^f$ and $b\geqslant2$. Their orders are $|Q|=q_0^{4b^2}\prod_{i=1}^{2b}(q_0^{2i}-1)$ and $|T|=q_0^{4b}(q_0^{2b}-1)(q_0^{4b}-1)$. For $b=2$ and $q_0=2$, the prime 7 divides $|Q|$ but not $|T|$. Consider next $b=2$ and $q_0>2$. By Proposition \ref{prop:ppd}, $q_0^6-1$ has a primitive prime divisor $z$. 
The factor $q_0^6-1$ occurs in $|Q|$, so $z\mid |Q|$. 
The primitiveness of $z$ gives $z\nmid q_0^2-1$. Since $(q_0^6-1,q_0^4-1)=(q_0^6-1,q_0^8-1)=q_0^2-1$, we obtain $z\nmid q_0^4-1$ and $z\nmid q_0^8-1$. Hence $z\nmid |T|$.
Assume now that $b\geqslant 3$. By \cite[Theorem 418]{HW1960}, there is a prime $w$ satisfying $b<w<2b$. Since $(q_0,2w)\neq (2,6)$, Proposition \ref{prop:ppd} yields a primitive prime divisor $z$ of $q_0^{2w}-1$. The inequality $w<2b$ implies that $q_0^{2w}-1$ is a factor of $|Q|$, so $z\mid |Q|$. Since $2b<2w$, the primitiveness of $z$ gives $z\nmid q_0^{2b}-1$. Furthermore, $w$ is an odd prime greater than $b$, and so $w\nmid 2b$. Thus $(2w,4b)<2w$, and the identity $(q_0^{2w}-1,q_0^{4b}-1)=q_0^{(2w,4b)}-1$ implies that $z\nmid q_0^{4b}-1$. Then $z\nmid |T|$. This proves (i).

We now prove (ii). Suppose that $Q=TR$ with $R<Q$ solvable. First consider $Q=\A_n$. For all alternating cases in Table \ref{tab:groups}, we have $n\geqslant 12$, so the exceptional degrees $n=6, 8$ in \cite[Proposition 4.3]{LX2022} do not occur. By \cite[Proposition 4.3]{LX2022}, we obtain $\A_{n-j}\leqslant T$ for some $j\in\{1,2,3\}$, which implies $\A_{n-3}\leqslant T$. For $(Q,T)=(\A_{12},\M_{12})$, this would give $\A_9\leqslant \M_{12}$. However, $7\mid |\A_9|$ whereas $7\nmid |\M_{12}|$, a contradiction.

Now let $T=\POmega_8^+(q)$. When $q=2$ or 3, the prime $11$ divides $|\A_{n-3}|$ but not $|T|$, contradicting $\A_{n-3}\leqslant T$. Now assume that $q\geqslant 4$ is even or $q\geqslant 5$ is odd. Every prime divisor of $|T|$ is at most $q^6-1$. For even $q$, we have $n-3-2(q^6-1)=q^6(q-2)-q^3-1>0$, whereas for odd $q$, $n-3-2(q^6-1)=(q^6(q-4)-q^3-2)/2>0$. Hence $2(q^6-1)<n-3$. By \cite[Theorem 418]{HW1960}, there is a prime $z$ satisfying $q^6-1<z<2(q^6-1)<n-3$. Since $z>2$, it follows that $z\mid |\A_{n-3}|$, while the bound on the prime divisors of $|T|$ gives $z\nmid |T|$. This again contradicts $\A_{n-3}\leqslant T$.

Next consider $T=\Sp_4(q)$, where $q=2^a\geqslant 4$. The equality $|T|=q^4(q^2-1)^2(q^2+1)$ shows that every prime divisor of $|T|$ is at most $q^2+1$. A direct calculation gives $n-3-2(q^2+1)=(q^4-5q^2-10)/2>0$, so $2(q^2+1)<n-3$. By \cite[Theorem 418]{HW1960}, there is a prime $z$ such that $q^2+1<z<2(q^2+1)<n-3$. The inequality $z<n-3$ and the fact that $z>2$ imply $z\mid |\A_{n-3}|$, whereas $z>q^2+1$ gives $z\nmid |T|$. This is incompatible with $\A_{n-3}\leqslant T$.

It remains to consider the two symplectic pairs. For $(Q,T)=(\Sp_8(2),\POmega_8^+(2))$, \cite[Theorem 1.1 (d)]{LX2022} applied to $Q=RT$ forces $T\cong \Omega^-_8(2)$. This is impossible because $17\mid |\Omega^-_8(2)|$ but $17\nmid |\POmega_8^+(2)|$.
Consider now $Q=\Sp_{4b}(q_0)$ and $T=\Sp_4(q_0^b)$, where $q_0=2^f$ and $b\geqslant 2$. If $Q=RT$ with $R$ solvable, then the same symplectic family in \cite[Theorem 1.1 (d)]{LX2022} requires $T\cong \Omega^-_{4b}(q_0)$. 
However, $v_2(|\Omega^-_{4b}(q_0)|)=2b(2b-1)f\neq 4bf=v_2(|T|)$ for $b\geqslant 2$, a contradiction. Thus (ii) holds, completing the proof of the lemma. \qed

\begin{lemma}\label{lem:PA1}
Suppose that $X$ is of product action type and that $\Ga$ is neither $X$-normal nor complete. Then $\soc(X)\cong Q^k$ and $(Q,T,k)$ is listed in Table \ref{tab:exact-pa}.
\end{lemma}

\pf\, Suppose that $X$ is of product action type. Then, by \cite[Section 2]{P1993}, $X$ is a subgroup of $H\wr\Sy_r$, where $H$ is a quasiprimitive permutation group that is either of almost simple type or simple diagonal action type. 
Let $L=\soc(X)=Q_1\times\cdots\times Q_r\cong Q^r$, where each $Q_i\cong Q$ is nonabelian simple and $r\geqslant 2$ is an integer. By Proposition \ref{subgroup}, $G\leqslant L$.
Since $\Ga$ is connected $(X,2)$-arc-transitive, the simple diagonal case is excluded by \cite[Theorem 2]{P1993}. Hence $H$ is of almost simple type. 
The argument of \cite[Section 2]{P1993} shows that for $\a\in V(\Ga)$, $L_\a$ is a subdirect subgroup of $R_1\times\cdots\times R_r$ where $1\neq R_i<Q_i$ and $R_i\cong R_1$ for all $1\leqslant i\leqslant r$. 
Let $\pi_i:L\to Q_i$ be the natural projection for $1\leqslant i\leqslant r$. Then $R_i=\pi_i(L_\a)$. We proceed in four steps to complete the proof of the lemma.

\vspace{0.5em}
\emph{Step 1. Every restriction $\pi_i|_{L_\a}$ is injective.}
Since $G$ is regular on $V(\Ga)$ and $G\leqslant L$, we have $L=GL_\a$ and $G\cap L_\a=1$. Moreover, $\Ga$ is non-bipartite; otherwise, the action of the transitive group $G$ on the two parts of the unique bipartition would give a subgroup of index 2 in $G\cong T^k$, which is impossible. 
Let $\{\a,\b\}$ be an edge of $\Ga$. By \cite[Lemma 2.1]{Lu2026}, $X=\l X_\a, X_\b\r$.
Hence by \cite[Theorem 1.1]{Lu2026}, either every $\pi_i|_{L_\a}$ is injective, or $L_\a$ is solvable and $\pi_i|_{L_\a}$ is not injective for some $i$.

Assume that $\pi_i|_{L_\a}$ is not injective for some $i$. The quasiprimitivity of $X$ makes every nontrivial normal subgroup of $X$ transitive, so $\Ga$ is basic with respect to $X$ in the sense of \cite{LS2023}. The graph is connected $(X,2)$-arc-transitive with valency at least 3. Since $\Ga$ is non-bipartite, it is not complete bipartite, and $L\cong Q^r$ is not simple as $r\geqslant 2$. 
By \cite[Theorem 4.2]{LS2023}, every $Q_i$ is semiregular on $V(\Ga)$. If $r=2$, then $\ker(\pi_1|_{L_\a})=L_\a\cap Q_2=(Q_2)_\a=1$ and $\ker(\pi_2|_{L_\a})=L_\a\cap Q_1=(Q_1)_\a=1$, contrary to the assumption. Hence $r\geqslant 3$. By \cite[Theorem 1.1 (2)]{Lu2026}, $L_\a\cong(\ZZ_p^t\times\ZZ_{m_1}).\ZZ_m$ and $|\Ga(\a)|=p^t$, where $m_1\mid m$ and $m\mid p^e-1$ for some proper divisor $e$ of $t$.

Suppose that $T\not\cong Q$. If $Q_i\leqslant G$ for some $i$, then $Q_i\unlhd G$, and Proposition \ref{obs} gives $Q_i\cong T$, contrary to $T\not\cong Q\cong Q_i$. We next show that $G\ntrianglelefteq X$. Suppose to the contrary that $G\unlhd X$. Since $L$ is the unique minimal normal subgroup of the quasiprimitive group $X$ and $G\neq1$, we have $L\leqslant G$. Together with $G\leqslant L$, this gives $G=L$, forcing $T\cong Q$, a contradiction. Hence $G\ntrianglelefteq X$.
By Lemma \ref{lem:framework} (1)(c), one of the following holds.
\begin{enumerate}[leftmargin=1.2em,label=(\roman*)]
\item [(a)] $k=r$ and $T_i<Q_i$ for every $i$.
\item [(b)] The subgroups $T_i$ are pairwise disjoint strips in $L$, with $|\mathcal{P}(T_i)|\in\{1,2\}$ for every $i$, $|\mathcal{P}(T_j)|=2$ for some $j$ and $(T,Q)$ is listed in Table \ref{tab:groups}.
\item [(c)] $L\cong(\A_n)^r$, $L_\a\cong(\A_{n-1})^r$, $|V(\Ga)|=n^r$ and $n\geqslant10$.
\end{enumerate}

Assume that (a) holds and let $c=|Q:T|$. Since $L=GL_\a$, $G\cap L_\a=1$ and $k=r$, we have $|L_\a|=|L:G|=c^r=p^tm_1m$. The divisibilities $m_1\mid m$ and $m\mid p^e-1$ give $(p,m_1m)=1$, so $t=rv_p(c)$ and $r\mid t$. The action of $Q$ on $[Q:T]$ is faithful. Hence $Q\leqslant \Sy_c$. The simplicity of $Q$ excludes $c\leqslant 4$, while $c=5$ would give $Q\cong\A_5$, which has no proper nonabelian simple subgroup. Hence $c\geqslant 6$.
By \cite[Theorem 1.1 (2)(i)]{Lu2026}, $z$ is a primitive prime divisor of $p^t-1$ with $z\mid r$, Proposition \ref{prop:ppd} gives $t\mid z-1$. Hence $z>t$, whereas $z\mid r$ and $r\mid t$ give $z\leqslant r\leqslant t$, a contradiction. It remains to consider $(p,t)=(2,6)$. In this case $z\in{3,7}$, while $z\mid r$, $r\mid 6$ and $r\geqslant 3$. Hence $z=3$ and $r\in\{3, 6\}$. Since $e$ is a proper divisor of 6, we have $e\in\{1, 2, 3\}$, and $m\mid 2^e-1$ gives $m\in\{1, 3, 7\}$. Since $m_1\mid m$, every odd prime divisor of $m_1m$ occurs with exponent at most 2. The equality $c^r=2^6m_1m$ and $r\geqslant 3$ then show that $c$ has no odd prime divisor. Hence $m_1m=1$ and $c^r=2^6$, giving $c\in\{2, 4\}$, contrary to $c\geqslant 6$. Thus \cite[Theorem 1.1 (2)(i)]{Lu2026} is excluded.
If \cite[Theorem 1.1 (2)(ii)]{Lu2026} holds, that is, $(p,t)=(2,6)$ and $L$ is regular on the edges or arcs of $\Ga$. In the edge-regular case, $|L_\a|=|\Ga(\a)|/2=32$, contradicting $2^6\mid |L_\a|$. In the arc-regular case, $|L_\a|=|\Ga(\a)|=64$, so $c^r=64$. Since $r\geqslant 3$, this gives $c\leqslant 4$, again impossible. 
Finally, \cite[Theorem 1.1 (2)(iii)]{Lu2026} gives $t=2$, contradicting $r\mid t$ and $r\geqslant 3$. 
Hence none of \cite[Theorem 1.1 (2)(i)--(iii)]{Lu2026} occurs and so the above case (a) is excluded.

Assume that case (b) holds. Recall that the $T_i$ are pairwise disjoint strips in $L$ and $(T,Q)$ is listed in Table \ref{tab:groups}. Let $Q_j\in\{Q_1, \ldots, Q_r\}$ for any $j\in\{1,\cdots,r\}$. Suppose that $Q_j\notin\cup_{i=1}^{k}\mathcal{P}(T_i)$. Then $\pi_j(G)=1$, and projecting $L=GL_\a$ onto $Q_j$ gives $Q_j=R_j$, this is impossible as $R_j<Q_j$. Hence $Q_j\in\cup_{i=1}^{k}\mathcal{P}(T_i)$. As $j$ is arbitrary, $\{Q_1,\ldots,Q_r\}\subseteq\cup_{i=1}^{k}\mathcal{P}(T_i)$. The reverse inclusion is clear, so $\cup_{i=1}^{k}\mathcal{P}(T_i)=\{Q_1,\ldots,Q_r\}$.
The pairwise disjointness of the supports shows that, for each $j$, there is a unique $i$ such that $Q_j\in\mathcal{P}(T_i)$. In particular, $\pi_j(G)=\pi_j(T_i)\cong T$. Projecting $L=GL_\a$ onto $Q_j$ gives $Q_j=\pi_j(G)R_j$. Thus $Q=TR$ for a solvable proper subgroup $R<Q$, which is excluded by Lemma \ref{lem:strip-pair-exclusion} (ii).

Suppose that case (c) holds, that is, $L_\a\cong(\A_{n-1})^r$. Then $L_\a$ is nonsolvable, contrary to the solvability of $L_\a$.  Thus all of cases (a), (b) and (c) are impossible, and so is $T\not\cong Q$.

Suppose now that $T\cong Q$. For $L=GL_\a$ and $G\cap L_\a=1$, we have $|L_\a|=|Q|^{r-k}$. It follows that $r>k$ as $L_\a\neq 1$. Note that the kernel of $\pi_i|_{T_j}$ is normal in $T_j$ for every $j\in\{1,\cdots,k\}$. Hence either $\pi_i(T_j)=1$ or $\pi_i|_{T_j}$ is injective. In the latter case $|\pi_i(T_j)|=|T_j|=|Q_i|$, so $\pi_i(T_j)=Q_i$.
We claim that the supports $\mathcal{P}(T_1),\ldots,\mathcal{P}(T_k)$ are pairwise disjoint. Suppose that $Q_i\in\mathcal{P}(T_j)\cap \mathcal{P}(T_m)$ for distinct $j$ and $m$. Then $\pi_i(T_j)=\pi_i(T_m)=Q_i$. Since $[T_j,T_m]=1$, it follows that $[\pi_i(T_j),\pi_i(T_m)]=1$, and so $[Q_i,Q_i]=1$, which is impossible. Hence the supports are pairwise disjoint.
Moreover, $\bigcup_{j=1}^{k}\mathcal{P}(T_j)=\{Q_1,\ldots,Q_r\}$; otherwise, $\pi_i(G)=1$ for some $i$, and projection of $L=GL_\a$ onto $Q_i$ gives $Q_i=R_i$, contrary to $R_i<Q_i$. Hence $\sum_{j=1}^{k}|\mathcal{P}(T_j)|=r$. 
Given that $r>k$, there must exist some $j$ with $|\mathcal{P}(T_j)|\geqslant 2$. After relabeling if necessary, assume that $Q_1,\,Q_2\in\mathcal{P}(T_j)$.
The restrictions $\pi_1|_{T_j}$ and $\pi_2|_{T_j}$ are isomorphisms. Let $\bar\pi:L\to Q_1\times Q_2$ be the natural projection. Then $\bar\pi(G)=\bar\pi(T_j)$ is a full diagonal subgroup of $Q_1\times Q_2$. Thus $\bar\pi(G)=\{(x,x^\varphi)\mid x\in Q_1\}$ for some isomorphism $\varphi:Q_1\to Q_2$. Since $L=GL_\a$, we have $Q_1\times Q_2=\bar\pi(G)\bar\pi(L_\a)$. 
For any $q\in Q_1$, there exist $x\in Q_1$ and $(r_1,r_2)\in\bar\pi(L_\a)$ such that $(q,1)=(x,x^\varphi)(r_1,r_2)$. Since $\pi_i(L_\a)=R_i$ for $i=1,2$, we have $r_i\in R_i$. 
Comparing the two coordinates gives $x^\varphi=r_2^{-1}$ and $q=xr_1$, and so $x=(r_2^{-1})^{\varphi^{-1}}\in R_2^{\varphi^{-1}}$ and $q\in R_2^{\varphi^{-1}}R_1$. 
Hence $Q_1=R_2^{\varphi^{-1}}R_1$. Since $|R_2^{\varphi^{-1}}|=|R_1|$ and both are solvable proper subgroups of $Q_1$, this contradicts \cite[Lemma 5.2]{BPS2004}, whose four possible factorizations have nonsolvable proper factors. Thus $T\cong Q$ is impossible.

Every restriction $\pi_i|_{L_\a}$ is injective, completing Step 1.

\vspace{0.5em}
\emph{Step 2. We prove that $r=k$, $T_i<Q_i$ for every $i$ and $(Q,T,k)$ occurs in Tables \ref{tab:equal-prime-r}.}
By Step 1, each restriction $\pi_i|_{L_\a}$ is an isomorphism from $L_\a$ onto $R_i$. Hence $|L_\a|=|R_i|<|Q|$ for every $i$. The relations $L=GL_\a$ and $G\cap L_\a=1$ give $|Q|^r=|T|^k|R_i|$.
Suppose that $T\cong Q$. Then $|R_i|=|Q|^{r-k}$. The subgroup $R_i$ is nontrivial, so $r>k$. This yields $|R_i|\geqslant |Q|$, contrary to $R_i<Q_i\cong Q$. Thus $T\not\cong Q$.
The inclusion $R_i\leqslant Q_i$ implies $|R_i|\mid |Q|$. The identity $|T|^k=|Q|^r/|R_i|$ gives $|Q|^{r-1}\mid |T|^k$. Since $r\geqslant 2$, every prime divisor of $|Q|$ divides $|T|$, and so $\Pi(|Q|)\subseteq\Pi(|T|)$. On the other hand, $G\leqslant L$ gives $|T|^k\mid |Q|^r$, and so $\Pi(|T|)\subseteq\Pi(|Q|)$. Thus $\Pi(|Q|)=\Pi(|T|)$.

We next show that no minimal normal subgroup of $L$ is contained in $G$. Suppose that $Q_i\leqslant G$ for some $i$. Then $Q_i\unlhd G$, and Proposition \ref{obs} gives $Q_i\cong T$, contrary to $T\not\cong Q$. Hence no minimal normal subgroup of $L$ lies in $G$. 
It remains to prove that $G\ntrianglelefteq X$. Suppose that $G\unlhd X$. Then the unique minimal normal subgroup $L$ of $X$ is contained in $G$. 
The inclusion $G\leqslant L$ forces $G=L$, and so $T\cong Q$, a contradiction. Thus $G\ntrianglelefteq X$ and Lemma \ref{lem:framework} (1)(c)(ii)--(iv) now applies. Lemma \ref{lem:framework} (1)(c)(iii) cannot occur, because every nonisomorphic pair $(T,Q)$ in Table \ref{tab:groups} satisfies $\Pi(|T|)\neq \Pi(|Q|)$ by Lemma \ref{lem:strip-pair-exclusion} (i). 
For Lemma \ref{lem:framework} (1)(c)(iv), $Q\cong\A_n$, $L_\a\cong(\A_{n-1})^r$ and $n\geqslant 10$. Noting that  $r\geqslant 2$ and $|\A_{n-1}|>n$, we have $|L_\a|\geqslant |\A_{n-1}|^2>n|\A_{n-1}|=|\A_n|=|Q|$, contrary to $|L_\a|=|R_i|<|Q|$. Hence Lemma \ref{lem:framework} (1)(c)(iv) is also impossible and so Lemma \ref{lem:framework} (1)(c)(ii) holds, that is, $r=k$ and, after relabeling, $T_i<Q_i$ for every $i$. 
Let $d=|Q:T|$. The relations $L\cong Q^r$ and $G\cong T^r$ give $|L|/|G|=|Q|^r/|T|^r=d^r$. 
Since $G\leqslant L$ is regular on $V(\Ga)$, the transitivity of $L$ gives $L=GL_\a$, while regularity gives $G\cap L_\a=1$. Hence $|L:G|=|L_\a|=d^r$.
Together with $\Pi(|T|)=\Pi(|Q|)$, all the hypotheses of Lemma \ref{lem:equal-prime-divisors} are satisfied. Hence $(Q,T,k)$ occurs in Tables \ref{tab:equal-prime-r}. This completes Step 2.

\vspace{0.5em}
\emph{Step 3. Construction of the factorizations $Q_i=T_iB_i$ and $R_i=A_iB_i$.}
By Step 2, $r=k$ and $G=T_1\times\cdots\times T_r$ with $T_i<Q_i$ for every $i$. 
Step 1 shows that $\pi_i|_R:L_\a\to R_i$ is an isomorphism. The equalities $L=GL_\a$ and $G\cap L_\a=1$ yield $|L_\a|=|L:G|=d^r$, and so $|R_i|=d^r$. 
Projecting $L=GL_\a$ onto $Q_i$ gives $Q_i=T_iR_i$. For each $i$, define $H_i=\{x\in L_\a\mid \pi_i(x)\in T_i\}$ and $A_i=T_i\cap R_i$. 
The factorization $Q_i=T_iR_i$ gives $|A_i|=|T_i||R_i|/|Q_i|=d^{r-1}$, while $\pi_i(H_i)=A_i$. Injectivity of $\pi_i|_{L_\a}$ yields $|H_i|=d^{r-1}$ and $|L_\a:H_i|=d$. Define $K_i=\bigcap_{j\neq i}H_j$ and $B_i=\pi_i(K_i)$. 
The indices $|L_\a:H_j|=d$ give $|L_\a:K_i|\leqslant d^{r-1}$, so $|K_i|\geqslant d$, and so $|B_i|=|K_i|\geqslant d$. Moreover, $K_i\cap H_i=\bigcap_{j=1}^{r}H_j=L_\a\cap G=1$. If $b\in T_i\cap B_i$, choose $x\in K_i$ with $b=\pi_i(x)$. 
The condition $b\in T_i$ forces $x\in H_i$, whence $x=1$ and $b=1$. Thus $T_i\cap B_i=1$. 
The inclusion $T_iB_i\subseteq Q_i$ now gives $|T_i||B_i|=|T_iB_i|\leqslant |Q_i|=d|T_i|$, so $|B_i|\leqslant d$. Together with $|B_i|\geqslant d$, this yields $|B_i|=d$, and so $|T_iB_i|=|Q_i|$. Thus $Q_i=T_iB_i$.

Recall that $A_i=T_i\cap R_i$ and $|A_i|=d^{r-1}$. Since $A_i\leqslant T_i$ and $T_i\cap B_i=1$, we obtain $A_i\cap B_i=1$. Both $A_i$ and $B_i$ lie in $R_i$, and $|A_iB_i|=|A_i||B_i|=d^{r-1}d=d^r=|R_i|$, so $R_i=A_iB_i$. If $r=2$, then $K_1=H_2$ and $K_2=H_1$. The injectivity of $\pi_i|_{L_\a}$ gives $A_i\cong H_i$ and $B_i\cong K_i$, and so $A_1\cong B_2$ and $A_2\cong B_1$. This completes Step 3.

\vspace{0.5em}
\emph{Step 4. Determine the possible triples $(Q,T,k)$.}
By Step 2, the candidates $(Q,T,k)$ are listed in Table \ref{tab:equal-prime-r}. By Step 3, $Q_i=T_iB_i$ and $T_i\cap B_i=1$ for every $i\in\{1,\ldots,k\}$. Consider first row 1 of Table \ref{tab:equal-prime-r}, where $Q\cong \A_c$ and $T\cong \A_h$ with $5\leqslant h<c$. Fix $i\in\{1,\ldots,k\}$. If $B_i$ is solvable, then \cite[Theorem 3 and Table 4]{BL2021} gives either $T_i=\A_{c-1}$ or $\A_{c-2}\unlhd T_i\leqslant (\Sy_{c-2}\times \Sy_2)\cap \A_c$. The simplicity of $T_i$ in the latter case forces $T_i=\A_{c-2}$. When $B_i$ is nonsolvable, \cite[Theorem 1.2 and Table 1.1]{LWX2023} gives $T_i\cong \A_{c-1}$, $\A_{c-2}$ or $\A_{c-3}$, apart from $(Q_i,T_i)\cong (\A_{11},\A_7)$ or $(\A_{12},\A_7)$. These two exceptional pairs are excluded by $\Pi(|Q|)=\Pi(|T|)$, because $11\mid |Q|$ but $11\nmid |T|$. Thus $h=c-j$ for some $j\in\{1,2,3\}$, giving the alternating-group possibilities in Table \ref{tab:exact-pa}.

\begin{table}[htbp]
\centering
\renewcommand{\arraystretch}{1.25}
\setlength{\tabcolsep}{7pt}
{\scriptsize
\caption{Necessary product-action possibilities in Theorem \ref{main}.}
\label{tab:exact-pa}
\begin{tabular}{|c|c|c|p{7.2cm}|}
\hline
$Q$ & $T$ & $k$ & remark \\
\hline
$\A_c$           & $\A_{c-j}$   & $k\geqslant 2$ & $j\in\{1,2,3\}$, $c-j\geqslant 5$, $(c-j,c]$ prime-free, $d=c!/(c-j)!$, $kv_p(d)\leqslant v_p(|Q|)$ for $p\mid d$, and $d^kc\leqslant c!/2$. \\
\hline
$\POmega_8^+(2)$ & $\Sp_6(2)$   & 2 & \\
\hline
$\M_{24}$        & $\M_{23}$    & 2, 3 & \\
\hline
\end{tabular}}
\end{table}

For row 2 of Table \ref{tab:equal-prime-r}, we have $Q\cong \PSp_{2m}(q)$, $T\cong \POmega^-_{2m}(q)$ and $k=2$, where $m$ and $q$ are even. A solvable $B_i$ is excluded by \cite[Theorem 2.12]{FLLWXZ2025}, which requires $m$ to be odd. If $B_i$ is nonsolvable, \cite[Table 1]{LWX2023} contains no factorization $Q_i=T_iB_i$ with $Q_i\cong \PSp_{2m}(q)$ and $T_i\cong \POmega^-_{2m}(q)$. Hence row 2 is impossible. In row 3 of Table \ref{tab:equal-prime-r}, $Q_i\cong \POmega_{2m+1}(q)$ and $T_i\cong \POmega^-_{2m}(q)$, where $m\geqslant 4$ is even and $q$ is odd. Neither \cite[Theorem 3 and Table 4]{BL2021} for solvable $B_i$ nor \cite[Theorem 1.2 and Table 1]{LWX2023} for nonsolvable $B_i$ contains such a factorization. Hence row 3 is also excluded.

For row 4 of Table \ref{tab:equal-prime-r}, $Q_i\cong \POmega^+_{2m}(q)$ and $T_i\cong \POmega_{2m-1}(q)$, where $m\geqslant 4$ is even. The solvable-factor classification in \cite[Theorem 3 and Table 4]{BL2021} gives no factorization of this form. For nonsolvable $B_i$, \cite[Theorem 1.2 and Table 1]{LWX2023} leaves only $Q_i\cong \POmega_8^+(2)$, $T_i\cong \Sp_6(2)$ and $B_i\cong \Sy_5$. Hence $(Q,T,k)=(\POmega_8^+(2),\Sp_6(2),2)$, giving row 2 of Table \ref{tab:exact-pa}. 
Row 5 of Table \ref{tab:equal-prime-r} is excluded by the last assertion of Lemma \ref{lem:equal-prime-divisors}, because $G\cap L_\a=1$. For row 6 of Table \ref{tab:equal-prime-r}, $(Q,T,k)=(\M_{11},\PSL_2(11),2)$ and $d=12$. Step 3 gives $\M_{11}=T_iB_i$, where $T_i\cong \PSL_2(11)$, $|B_i|=12$ and $T_i\cap B_i=1$, but no such product occurs in \cite[Table 5]{G2006}. 
Hence row 6 of Table \ref{tab:equal-prime-r} is impossible. In row 8 of Table \ref{tab:equal-prime-r}, $(Q,T)=(\M_{24},\M_{23})$, $k\in\{2,3\}$ and $|B_i|=24$. Every group of order $24$ is solvable, and \cite[Theorem 3 and Table 4]{BL2021} gives $B_i\cong \Sy_4$, $\D_{24}$, $\D_8\times\ZZ_3$, $\ZZ_3:\D_8$ or $\A_4\times\ZZ_2$. Hence $(Q,T,k)=(\M_{24},\M_{23},2)$ or $(\M_{24},\M_{23},3)$, giving row 3 of Table \ref{tab:exact-pa}.

Finally, row 7 of Table \ref{tab:equal-prime-r} gives $(Q,T,k)=(\M_{12},\M_{11},2)$ with $d=12$. By Step 3, there exists a subgroup $R_i\leqslant \M_{12}$ of order $144$ such that $R_i=A_iB_i$, where $|A_i|=|B_i|=12$ and $A_i\cap B_i=1$. 
The Magma computation in Appendix \ref{app:M12} shows that $\M_{12}$ has exactly two conjugacy classes of subgroups of order 144, and each representative has a unique conjugacy class of subgroups of order 12. 
Thus $A_i$ and $B_i$ are conjugate in $R_i$. Let $\Omega$ be the set of right cosets of $B_i$ in $R_i$, so $|\Omega|=12$. The conjugacy of $A_i$ and $B_i$ implies that $A_i$ fixes a point of $\Omega$. On the other hand, $R_i=A_iB_i=B_iA_i$ implies that $A_i$ is transitive on $\Omega$, a contradiction. Thus row 7 is excluded. 
This completes Step 4 and the proof of the lemma. \qed

\begin{lemma}\label{lem:TW1}
Suppose that $X$ is of twisted wreath type. Then $G=\soc(X)\unlhd X$, and $\Ga$ is $X$-normal.
\end{lemma}

\pf\, Since $X$ is of twisted wreath type, $\soc(X)$ is regular on $V(\Ga)$. By Proposition \ref{subgroup}, $G\leqslant \soc(X)$. As $G$ is transitive on $V(\Ga)$, the regularity of $\soc(X)$ forces $G=\soc(X)$. Thus $G=\soc(X)\unlhd X$, and $\Ga$ is $X$-normal.\qed

\vspace{0.5em}
The preceding lemmas determine all five quasiprimitive types. We now combine these results to prove Theorem \ref{thm:qprim}.

\textbf{Proof of Theorem \ref{thm:qprim}}

\pf\, Let $\Ga=\Cay(G,S)$ be connected and $(X,2)$-arc-transitive, where $G=T^k$ with $k\geqslant 2$, and suppose that $X$ is quasiprimitive on $V(\Ga)$. By \cite[Theorem 1]{P1993}, $X$ is of affine, almost simple, simple diagonal, product action or twisted wreath type. Lemma \ref{lem:affine1} rules out the affine case, while Lemma \ref{lem:diagonal1} excludes the simple diagonal case. If $X$ is almost simple, Lemma \ref{lem:AS1} gives either $G\unlhd X$ or $\Ga$ is complete. If $X$ is of product action type and $\Ga$ is neither $X$-normal nor complete, Lemma \ref{lem:PA1} gives $\soc(X)\cong Q^k$ with $(Q,T,k)$ listed in Table \ref{tab:exact-pa}. In the twisted wreath case, Lemma \ref{lem:TW1} gives $G=\soc(X)\unlhd X$. These are precisely the possibilities stated in the theorem.  \qed

\section{The non-quasiprimitive case}\label{sec5}

Assume throughout this section that Hypothesis \ref{hypothesis} holds and that $X$ is not quasiprimitive on $V(\Ga)$. Let $N$ be a maximal intransitive normal subgroup of $X$, $\overline{X}=X/N$, $\overline{G}=GN/N$ and $\overline{\a}=\a^N$.  Now set $L=\soc(\overline{X})=Q_1\times\cdots\times Q_r\cong Q^r$ for some $r\geqslant1$, where each $Q_i\cong Q$ is simple.  We first state the result for the non-quasiprimitive case. 

\begin{theorem}\label{thm:non-quasiprimitive}
Suppose that Hypothesis \ref{hypothesis} holds and $X$ is not quasiprimitive on $V(\Ga)$. Then $\Ga$ is an $N$-cover of $\Ga_N$, $\overline{X}$ is quasiprimitive on $V(\Ga_N)$ and one of the following holds.
\begin{enumerate}[leftmargin=1.2em]
\item $\overline{X}$ is of affine type, $T\cong \PSL_2(7)$ and $\Ga_N\cong \K_{2^{3\ell}}$.
\item $\overline{X}$ is of almost simple type, either $\Ga_N$ is complete, or $\ell=1$ and one of the following holds: $\overline{G}=\soc(\overline{X})$; $\Ga_N$ is $(\soc(\overline{X}),2)$-arc-transitive; or one of the possibilities in Tables \ref{tab:AS1}, \ref{tab:AS2}, \ref{tab:AS3}, \ref{tab:AS4} and \ref{tab:AS5} occurs.
\item $\overline{X}$ is of product action type with $\soc(\overline{X})\cong Q^\ell$, either $\overline{G}\unlhd\overline{X}$, or $(Q,T,\ell)$ is listed in Table \ref{tab:equal-prime-r}. Moreover, if $N\leqslant G$, then $(Q,T,r)$ occurs in Table \ref{tab:exact-pa}.
\item $\overline{X}$ is of twisted wreath action type and $\Ga_N$ is a Cayley graph of $\overline{G}=\soc(\overline{X})$.
\end{enumerate}
\end{theorem}

We first record the normal quotient structure that will be used throughout the proof.

\begin{lemma}\label{lem:normal-quotient}
Suppose that Hypothesis \ref{hypothesis} holds and $X$ is not quasiprimitive on $V(\Ga)$. Then $\Ga$ is an $N$-cover of the non-bipartite graph $\Ga_N$, and $\overline{X}$ is quasiprimitive on $V(\Ga_N)$. Moreover, $\overline{G}\cong T^\ell$ for some $1\leqslant\ell\leqslant k$, and $\overline{G}$ is regular on $V(\Ga_N)$ if and only if $N\leqslant G$.
\end{lemma}

\pf\, Let $\mathcal{B}$ be the set of $N$-orbits on $V(\Ga)$. The transitivity of $G$ on $V(\Ga)$ induces a transitive action on $\mathcal{B}$. If $|\mathcal{B}|=2$, then $G$ has a subgroup of index 2, which is impossible because $G\cong T^k$ and $T$ is nonabelian simple. Thus $N$ has at least three orbits. Recall that $\Ga$ is a connected $(X,2)$-arc-transitive graph. 
Proposition \ref{prop:NQ} shows that $N$ is semiregular, $\Ga$ is an $N$-cover of $\Ga_N$, and $\Ga_N$ is $(\overline{X},2)$-arc-transitive. Let $K$ be the kernel of the action of $X$ on $\mathcal{B}$. Then $N\leqslant K\unlhd X$, while $K$ is intransitive on $V(\Ga)$. The maximality of $N$ forces $K=N$. If $M/N$ is a nontrivial normal subgroup of $\overline{X}$, then $M>N$ and so $M$ is transitive on $V(\Ga)$. Thus $M/N$ is transitive on $V(\Ga_N)$, proving that $\overline{X}$ is a quasiprimitive permutation group on $V(\Ga_N)$. 

The subgroup $G\cap N$ is normal in $G$, so Proposition \ref{obs} gives $G/(G\cap N)\cong T^\ell$ for some $1\leqslant \ell\leqslant k$. Hence $\overline{G}=GN/N\cong T^\ell$, and $\overline{G}$ is transitive on $V(\Ga_N)$. 
Suppose that $\Gamma_N$ is bipartite. Since $\overline{G}\cong T^\ell$ is vertex-transitive, some element of $\overline{G}$ swaps the two parts of the bipartition. Then $\overline{G}$ has a subgroup of index 2, impossible for $T^\ell$. Hence $\Gamma_N$ is non-bipartite.
Noting that $N$ is semiregular and $G$ is regular on $V(\Ga)$ yields $|V(\Ga_N)|=|G|/|N|$ and $|\overline{G}|= |G|/|G\cap N|$. Thus $\overline{G}$ is regular on $V(\Ga_N)$ if and only if $|G\cap N|=|N|$, i.e.\, $N\leqslant G$.
This completes the proof.   \qed

\vspace{0.5em}

Now that $\overline{X}$ is quasiprimitive on $V(\Ga_N)$, we consider the five types in the classification of quasiprimitive permutation groups separately.

\begin{lemma}\label{lem:affine2}
If $\overline{X}$ is of affine type, then $T\cong \PSL_2(7)$ and $\Ga_N\cong \K_{2^{3\ell}}$.
\end{lemma}

\pf\, Suppose that $\overline{X}$ is of affine type. Then, by Lemma \ref{lem:framework} (2)(a), $T\cong\PSL_2(7)$ and $L=\soc(\overline{X})\cong\ZZ_2^{3\ell}$. The subgroup $\overline{G}\cap L$ is abelian and normal in $\overline{G}\cong T^\ell$, so $\overline{G}\cap L=1$. An affine quasiprimitive group is primitive, and the classification in \cite{IP1993} shows that $\Ga_N$ is a complete graph, a folded cube, a projective-space graph, or one of $\Ga(C_{23})$ and $\Ga(C_{22})$. In the complete case, the regularity of $L$ gives $|V(\Ga_N)|=|L|=2^{3\ell}$, and so $\Ga_N\cong \K_{2^{3\ell}}$.

Consider a folded cube $\Ga_N\cong\Box_n$. Its order gives $2^{n-1}=|L|=2^{3\ell}$, so $n-1=3\ell$. For $n\geqslant 5$, Example \ref{ex:folded-cube} gives $|\Aut(\Box_n)|=2^{n-1}n!$. The inclusion $\overline{G}\leqslant\Aut(\Box_n)$ and $7^\ell\mid |\overline{G}|$ imply $7^\ell\mid n!$. However, $v_7(n!)=\sum_{i\geqslant1}\lfloor n/7^i\rfloor<n/6=(3\ell+1)/6<\ell$, a contradiction. If $n=4$, then $|\Aut(\Box_4)|=2^7\cdot3^2$ by Example \ref{ex:folded-cube}, whereas $7\mid |\overline{G}|$. The value $n=3$ is incompatible with $n-1=3\ell$. Thus no folded cube occurs.

Now let $\Ga_N\cong P_m(a)$. Example \ref{ex:projective-space-graph} gives $m^a=3\ell$ and $\overline{X}/L\leqslant \PGammaL_m(2^a)$. Note that $\overline{G}=\overline{G}/(\overline{G}\cap L)\cong \overline{G}L/L\leqslant \overline{X}/L$ as $\overline{G}\cap L=1$. 
Moreover, $\PGammaL_m(2^a)/\PGL_m(2^a)$ is cyclic, while $\overline{G}\cong T^\ell$ is perfect. Hence $\overline{G}\lesssim\PGL_m(2^a)$. Since $v_2(|T|)=v_2(|\PSL_2(7)|)=3$, we obtain $3\ell\leqslant v_2(|\PGL_m(2^a)|)=am(m-1)/2$. Together with $m^a=3\ell$, this gives $m^{a-1}\leqslant a(m-1)/2$. 
If $a\geqslant 2$, then $m\geqslant 3$ and $m^{a-1}=(1+(m-1))^{a-1}\geqslant 1+(a-1)(m-1)>a(m-1)/2$, a contradiction. Hence $a=1$. Example \ref{ex:projective-space-graph} now gives $P_m(1)\cong \K_{2^m}$, and $m=3\ell$, so $\Ga_N\cong \K_{2^{3\ell}}$. Finally, Example \ref{ex:golay} gives $|V(\Ga(C_{23}))|=2^{11}$ and $|V(\Ga(C_{22}))|=2^{10}$, whereas $|V(\Ga_N)|=|L|=2^{3\ell}$. Neither $3\ell=11$ nor $3\ell=10$ is possible. Thus the two code graphs are excluded, and the result follows. \qed

\vspace{0.5em}

\begin{lemma}\label{lem:AS2}
Suppose that $\overline{X}$ is of almost simple type. Then either $\Ga_N$ is complete, or $\overline{G}=T\leqslant \soc(\overline{X})$ and one of the following holds:
\begin{enumerate}[label=(\roman*),leftmargin=2em]
\item $\overline{G}=\soc(\overline{X})$;
\item one of the possibilities in Tables \ref{tab:AS1}, \ref{tab:AS2}, \ref{tab:AS3}, \ref{tab:AS4} and \ref{tab:AS5} holds;
\item $\Ga_N$ is $(\soc(\overline{X}),2)$-arc-transitive.
\end{enumerate}
\end{lemma}

\pf\, Suppose that $\overline{X}$ is of almost simple type, and let $L=\soc(\overline{X})$. Recall that $\Ga$ is connected $(X,2)$-arc-transitive. Hence $\Ga_N$ is connected $(\overline{X},2)$-arc-transitive, and so $\overline{X}_{\overline{\a}}^{\Ga_N(\overline{\a})}$ is 2-transitive on $\Ga_N(\overline{\a})$. By \cite[Theorem 3.21]{PS2018}, $\soc(\overline{X}_{\overline{\a}}^{\Ga_N(\overline{\a})})$ is either a regular elementary abelian $p$-group, or a nonregular nonabelian simple group.
Lemma \ref{lem:framework} (2)(b) shows that $\Ga_N$ is complete whenever $\ell\geqslant 2$. We may therefore assume that $\Ga_N$ is not complete. Then $\ell=1$ and so $\overline{G}=T$. Since $\overline{X}$ is quasiprimitive on $V(\Ga_N)$ with nonabelian socle $L$, and $\overline{G}=T$ is transitive on $V(\Ga_N)$, Proposition \ref{subgroup} gives $T\leqslant L$. If $T=L$, then (i) holds. Hence only the case $T<L$ remains.

Assume first that $N\leqslant G$.  Then Lemma \ref{lem:normal-quotient} gives that $\overline{G}=T$ is regular on $V(\Ga_N)$. Thus $\Ga_N$ is a connected $(\overline{X},2)$-arc-transitive Cayley graph of the nonabelian simple group $T$. By \cite[Theorem 1.8]{LWX2022} and its proof, one of the following holds: $\Ga_N$ is complete, $\Ga_N$ is $\overline{X}$-normal, or $(L,T,L_{\overline{\a}})$ occurs in Table \ref{tab:AS1}. The first case has already been excluded. If $\Ga_N$ is $\overline{X}$-normal, then $T\unlhd\overline{X}$. Since $\overline{X}$ is almost simple with unique minimal normal subgroup $L$ and $1<T\leqslant L$, we obtain that $L\leqslant T\leqslant L$, and so $T=L$, contrary to $T<L$. Hence $(L,T,L_{\overline{\a}})$ is listed in Table \ref{tab:AS1} when $N\leqslant G$.

\begin{table}[htbp]
\centering
\setlength{\tabcolsep}{4pt}
\renewcommand{\arraystretch}{1.3}
\caption{The triples $(L,T,L_{\overline{\a}})$ arising from the regular action of $T$.}
\label{tab:AS1}
{\scriptsize
\begin{tabular}{|c|c|c|c|c|c|}
\hline
$L$ & $T$ & $L_{\overline{\a}}$ &          $L$ & $T$ & $L_{\overline{\a}}$\\
\hline
$\A_q$  & $\A_{q-2}$ & $\leqslant \AGammaL_1(q)$, $q$ a prime power          &$\A_{25}$   & $\A_{23}$ &$\ZZ_5^2:\SL_2(3)$\\
\hline
$\A_{49}$ & $\A_{47}$ & $\ZZ_7^2:\Q_8.\Sy_3$  & $\A_{121}$        & $\A_{119}$ & $\ZZ_{11}^2:\SL_2(3).\ZZ_5,\ \ZZ_{11}^2:\SL_2(5)$\\
\hline
$\A_{529}$   & $\A_{527}$ & $\ZZ_{23}^2:\SL_2(3).\ZZ_{22}$          & $\A_{841}$   & $\A_{839}$ & $\ZZ_{29}^2:(\SL_2(5)\times\ZZ_7)$\\
\hline
$\A_{3481}$   & $\A_{3479}$ & $\ZZ_{59}^2:(\SL_2(5)\times\ZZ_{29})$          & $\Sp_{2m}(q)$ & $\Omega^-_{2m}(q)$ & $\leqslant\E_q^m:\GammaL_1(q^m)$, $q$ a power of 2\\
\hline
$\M_{12}$ & $\M_{11}$  & $\A_4$             & $\A_n$       & $\A_{n-1}$ & $[n]$\\
\hline
\end{tabular}}
\end{table}
From now on assume that $N\nleqslant G$. Lemma \ref{lem:normal-quotient} shows that $T$ is not regular on $V(\Ga_N)$. Thus $T_{\overline{\a}}\neq 1$, and the inclusion $T<L$ gives $L_{\overline{\a}}\neq 1$. Recall that $L=\soc(\overline{X})\unlhd\overline{X}$. \cite[Lemmas 4.1 and 4.2]{L2025} gives that $L_{\overline{\a}}$ is transitive on $\Ga_N(\overline{\a})$ and $\soc(L_{\overline{\a}}^{\Ga_N(\overline{\a})})=\soc(\overline{X}_{\overline{\a}}^{\Ga_N(\overline{\a})})$. 
Note that $L$ is transitive on $V(\Ga_N)$. Thus $\Ga_N$ is also $L$-arc-transitive.

We next consider the solvable and nonsolvable cases for $L_{\overline{\a}}$ separately.

\smallskip
\noindent{\textbf{Claim 1.} If $L_{\overline{\a}}$ is solvable, then $(L,T,L_{\overline{\a}})$ are listed in Tables \ref{tab:AS2}, \ref{tab:AS3}, \ref{tab:AS4} and \ref{tab:AS5}.}

Recall that $\soc(\overline{X}_{\overline{\a}}^{\Ga_N(\overline{\a})})$ is either a regular elementary abelian $p$-group, or a nonregular nonabelian simple group. 
The latter case is impossible because $L_{\overline{\a}}$ is solvable and $\soc(L_{\overline{\a}}^{\Ga_N(\overline{\a})})=\soc(\overline{X}_{\overline{\a}}^{\Ga_N(\overline{\a})})$. Hence $|\Ga_N(\overline{\a})|=p^e$ and $\soc(L_{\overline{\a}}^{\Ga_N(\overline{\a})})\cong\ZZ_p^e$ for some prime $p$ and integer $e\geqslant 1$.

Suppose first that $e=1$. Then $\Ga_N$ has prime valency. By the preceding discussion, $\Ga_N$ is connected $L$-arc-transitive, $\overline{G}=T<L=\soc(\overline{X})$ is transitive on $V(\Ga_N)$, $\Ga_N$ is not complete, and $L_{\overline{\a}}$ is solvable. Hence the hypotheses of \cite[Lemma 4.3]{DLY2026} are satisfied, and $(L,T,L_{\overline{\a}})$ occurs in rows 1--3 of \cite[Table 5]{DLY2026}.
Row 2 has $L\cong\M_{23}$, $T\cong \M_{22}$ and $L_{\overline{\a}}\cong\ZZ_{23}$, so $|L:T|=|L_{\overline{\a}}|=23$, forcing $T\cap L_{\overline{\a}}=1$. For row 3, $L\cong \PSL_2(p)$, $T\cong \A_5$ and $L_{\overline{\a}}\leqslant \ZZ_p:\ZZ_{(p-1)/2}$, where $p\in\{11,19,29\}$. The factorization $L=TL_{\overline{\a}}$ requires $|L:T|=p(p^2-1)/120$ to divide $|L_{\overline{\a}}|$. In particular, $L_{\overline{\a}}=\ZZ_p:\ZZ_a$ with $a\mid (p-1)/2$ and $(p^2-1)/120\mid a$. 
The only possibilities are $a=(p^2-1)/120$ and $a=(p-1)/2$. The first gives $T\cap L_{\overline{\a}}=1$, while the second gives $|V(\Ga_N)|=|L:L_{\overline{\a}}|=p+1$. Since the valency is $p$, the latter forces $\Ga_N\cong \K_{p+1}$. Thus neither row occurs in the present nonregular, noncomplete case. Thus $L\cong \A_m$, $T\cong \A_{m-1}$ and $L_{\overline{\a}}$ is transitive of degree $m$, where $m\geqslant 6$. 
The prime-valent local action gives $v_p(|L_{\overline{\a}}|)=1$ by \cite[Lemma 2.6]{DLY2026}. Since $m=|L:T|$ divides $|L_{\overline{\a}}|$, it follows that $v_p(m)\leqslant 1$. Moreover, $|L_{\overline{\a}}|>m$ because $T\cap L_{\overline{\a}}\neq 1$. This gives the first row of Table \ref{tab:AS2}.

Assume next that $e\geqslant 2$. We consider the possibilities for $L$ given by the classification of finite nonabelian simple groups. The solvability of $L_{\overline{\a}}$, together with \cite[Proposition 4.2]{LX2022}, excludes exceptional simple groups of Lie type. 
We first consider the case where $L\cong \A_m$. By \cite[Proposition 4.3 (a)]{LX2022}, we have $T=\A_{m-1}$ and $L_{\overline{\a}}$ is transitive of degree $m$. Since $L=TL_{\overline{\a}}$, we have $|L:T|=|L_{\overline{\a}}:T\cap L_{\overline{\a}}|=m$. The intersection $T\cap L_{\overline{\a}}$ is nontrivial, and so $|L_{\overline{\a}}|> m\geqslant 6$. This gives the first row of Table \ref{tab:AS2}.

By \cite[Proposition 4.3 (b)]{LX2022}, another possibility is $T=\A_{m-2}$ and $m=s^f$ for some prime $s$, where $L_{\overline{\a}}$ is 2-homogeneous of degree $m$. Moreover, either $L_{\overline{\a}}\leqslant \AGammaL_1(s^f)$ or $L_{\overline{\a}}$ is 2-homogeneous of degree $m$ and $(L_{\overline{\a}},m)$ occurs in the exceptional list given there. Since $|L:T|=m(m-1)=|L_{\overline{\a}}:T\cap L_{\overline{\a}}|$ and $T\cap L_{\overline{\a}}\neq 1$, we have $|L_{\overline{\a}}|>m(m-1)$. The case $L_{\overline{\a}}\leqslant \AGammaL_1(s^f)$ gives the second row of Table \ref{tab:AS2}.

It remains to consider the exceptional groups in \cite[Proposition 4.3 (b)]{LX2022}. For $m=25$, the subgroup $\ZZ_5^2:\SL_2(3)$ has order $25\cdot24=m(m-1)$ and is excluded. Thus $L_{\overline{\a}}\cong \ZZ_5^2:\Q_8.\ZZ_6$ or $\ZZ_5^2:\SL_2(3).\ZZ_4$. Let $V\cong \ZZ_5^2$ be the normal elementary abelian subgroup of $L_{\overline{\a}}$. Since $L_{\overline{\a}}$ is 2-homogeneous of degree 25, it is primitive, and every nontrivial normal subgroup of $L_{\overline{\a}}$ is transitive. Hence $V$ is regular and a minimal normal subgroup of $L_{\overline{\a}}$. Moreover, $V$ is the unique Sylow 5-subgroup of $L_{\overline{\a}}$.
Note that $\Ga_N$ is $L$-arc-transitive. Then there exists $g\in L$ such that $\overline{\b}=\overline{\a}^g$ for some $\overline{\b}\in \Ga_N(\overline{\a})$ and $|\Ga_N(\overline{\a})|=|L_{\overline{\a}}:L_{\overline{\a}}\cap L_{\overline{\a}}^g|=p^e$. The connectivity of $\Ga_N$ gives $L=\l L_{\overline{\a}}, g\r$. 
If $p\neq 5$, then $|L_{\overline{\a}}:L_{\overline{\a}}\cap L_{\overline{\a}}^g|=p^e$ is coprime to 5. Hence $L_{\overline{\a}}\cap L_{\overline{\a}}^g$ contains a Sylow 5-subgroup of $L_{\overline{\a}}$, which must be $V$ by uniqueness. Thus $V\leqslant L_{\overline{\a}}^g$ and so $V$ is also a Sylow 5-subgroup of $L_{\overline{\a}}^g$. The unique Sylow 5-subgroup of $L_{\overline{\a}}^g$ is $V^g$, and so $V=V^g$. Thus $g\in N_L(V)$ and $L_{\overline{\a}}\leqslant N_L(V)$, which gives $L=\l L_{\overline{\a}}, g\r\leqslant N_L(V)<L$, a contradiction. Hence $p=5$.
Since $V\cap L_{\overline{\a}}^{[1]}\unlhd L_{\overline{\a}}$ and $V$ is a minimal normal subgroup of $L_{\overline{\a}}$, either $V\cap L_{\overline{\a}}^{[1]}=1$ or $V\leqslant L_{\overline{\a}}^{[1]}$. The latter is impossible because $L_{\overline{\a}}/V$ is a 5'-group, whereas $\soc(L_{\overline{\a}}/L_{\overline{\a}}^{[1]})\cong \soc(L_{\overline{\a}}^{\Ga_N(\overline{\a})})\cong \ZZ_5^e$ and $e\geqslant 2$. 
Hence $V\cap L_{\overline{\a}}^{[1]}=1$. Therefore $VL_{\overline{\a}}^{[1]}/L_{\overline{\a}}^{[1]}\cong V$ is a nontrivial minimal normal subgroup of $L_{\overline{\a}}/L_{\overline{\a}}^{[1]}\cong L_{\overline{\a}}^{\Ga_N(\overline{\a})}$, and so $VL_{\overline{\a}}^{[1]}/L_{\overline{\a}}^{[1]}=\soc(L_{\overline{\a}}/L_{\overline{\a}}^{[1]})\cong \soc(L_{\overline{\a}}^{\Ga_N(\overline{\a})})\cong \ZZ_5^e$. Thus $e=2$, and $\Ga_N$ has valency 25. This gives the third row of Table \ref{tab:AS2}.

The remaining exceptional cases are similar. For $m=49$, the subgroup $\ZZ_7^2:\Q_8.\Sy_3$ has order $49\cdot48=m(m-1)$ and is excluded, while $\ZZ_7^2:\SL_2(3).\ZZ_6$ gives valency 49. For $m=81$, all three groups in \cite[Proposition 4.3 (b)]{LX2022} have order greater than $81\cdot80$, and the same argument with $V\cong\ZZ_3^4$ gives valency 81. For $m=121$, the subgroup $\ZZ_{11}^2:\SL_2(3).\ZZ_5$ has order $121\cdot120=m(m-1)$ and is excluded, while $\ZZ_{11}^2:\SL_2(3).\ZZ_{10}$ gives valency 121. The unique candidate for $m=529$ has order $529\cdot528=m(m-1)$ and is excluded. These give rows 4--6 of Table \ref{tab:AS2}.

\begin{table}[htbp]
\centering
\setlength{\tabcolsep}{4pt}
\renewcommand{\arraystretch}{1.3}
\caption{Alternating possibilities with $L_{\overline{\a}}$ solvable.}
\label{tab:AS2}
{\scriptsize
\begin{tabular}{|c|c|>{\centering\arraybackslash}p{3.8cm}|>{\centering\arraybackslash}p{5.0cm}|}
\hline
$L$ & $T$ & $L_{\overline{\a}}$ & remark\\
\hline
$\A_m$ & $\A_{m-1}$ & transitive of degree $m$ & $|L_{\overline{\a}}|>m\geqslant 6$ \\
\hline
$\A_m$ & $\A_{m-2}$ & $\leqslant\AGammaL_1(s^f)$ & $m=s^f\geqslant 7$ for some prime $s$,  $|L_{\overline{\a}}|>m(m-1)$\\
\hline
$\A_{25}$ & $\A_{23}$ & $\ZZ_5^2:\Q_8.\ZZ_6$, $\ZZ_5^2:\SL_2(3).\ZZ_4$ & $|\Ga_N(\overline{\a})|=5^2$\\
\hline
$\A_{81}$ & $\A_{79}$ & $\ZZ_3^4:2^{1+4}.\ZZ_5$, $\ZZ_3^4:2^{1+4}.\D_{10}$, $\ZZ_3^4:2^{1+4}.\AGL_1(5)$ & $|\Ga_N(\overline{\a})|=3^4$\\
\hline
$\A_{121}$ & $\A_{119}$ & $\ZZ_{11}^2:\SL_2(3).\ZZ_{10}$ & $|\Ga_N(\overline{\a})|=11^2$\\
\hline
$\A_{49}$ & $\A_{47}$ & $\ZZ_7^2:\SL_2(3).\ZZ_6$ & $|\Ga_N(\overline{\a})|=49$\\
\hline
\end{tabular}}
\end{table}

\cite[Proposition 4.3 (c)]{LX2022} gives $m\in\{8,32\}$ and $T=\A_{m-3}$. If $m=8$, then $L_{\overline{\a}}\cong \AGL_1(8)$ or $\AGammaL_1(8)$, and so $|L_{\overline{\a}}|\leqslant168< 8\cdot7\cdot6=|L:T|$, contrary to $|L:T|\mid |L_{\overline{\a}}|$. 
For $m=32$, we have $L_{\overline{\a}}\cong \AGammaL_1(32)$ and $|L_{\overline{\a}}|= 32\cdot31\cdot5< 32\cdot31\cdot30= |L:T|$, which gives the same contradiction. Thus this case does not occur.

Consider the case in \cite[Proposition 4.3(d)]{LX2022}, that is, $L\cong\A_6$, $T\cong\A_5$ and $L_{\overline{\a}}\leqslant (\Sy_4\times\Sy_2)\cap \A_6\cong \Sy_4$. 
Since $L=TL_{\overline{\a}}$, we have $|L_{\overline{\a}}:T\cap L_{\overline{\a}}|=|L:T|=6$. Together with $T\cap L_{\overline{\a}}\neq 1$ and $|L_{\overline{\a}}|\mid 24$, it follows that $|L_{\overline{\a}}|\in \{12, 24\}$.
First let $|L_{\overline{\a}}|=24$. Then $L_{\overline{\a}}\cong \Sy_4$ and $L_{\overline{\a}}=(\Sy_4\times \Sy_2)\cap \A_6$. Thus $L_{\overline{\a}}$ is the stabilizer of a 2-subset in the natural action of $\A_6$ on $\{1,\ldots,6\}$, and the action of $L$ on the cosets of $L_{\overline{\a}}$ is permutation equivalent to the action of $\A_6$ on the 15 two-subsets of $\{1,\ldots,6\}$. 
Fix the 2-subset $\{5,6\}$. Among the remaining 2-subsets, six are contained in $\{1,2,3,4\}$ and eight contain exactly one point of $\{5,6\}$. Hence the $L_{\overline{\a}}$-orbits have lengths 1, 6 and 8.
Since $L_{\overline{\a}}$ is transitive on $\Ga_N(\overline{\a})$, the set $\Ga_N(\overline{\a})$ is an $L_{\overline{\a}}$-orbit. As $|\Ga_N(\overline{\a})|=p^e$ with $e\geqslant 2$, we have $|\Ga_N(\overline{\a})|=8$. Let $\overline{\b}\in\Ga_N(\overline{\a})$. Then $|L_{\overline{\a}\overline{\b}}|=|L_{\overline{\a}}|/8=3$, and $L_{\overline{\a}}^{[1]}\unlhd L_{\overline{\a}}$ with $L_{\overline{\a}}^{[1]}\leqslant L_{\overline{\a}\overline{\b}}$. 
Since $L_{\overline{\a}}\cong\Sy_4$ has no nontrivial normal subgroup of order dividing 3, we obtain $L_{\overline{\a}}^{[1]}=1$. Hence $L_{\overline{\a}}^{\Ga_N(\overline{\a})}\cong\Sy_4$, whose socle is $\ZZ_2^2$. On the other hand, $L_{\overline{\a}}^{\Ga_N(\overline{\a})}$ is a solvable 2-transitive group of degree 8, so its socle is regular elementary abelian of order 8, isomorphic to $\ZZ_2^3$, a contradiction.

Now let $|L_{\overline{\a}}|=12$. Since $L_{\overline{\a}}\leqslant(\Sy_4\times\Sy_2)\cap\A_6\cong\Sy_4$, the subgroup $L_{\overline{\a}}$ has index 2 in $(\Sy_4\times\Sy_2)\cap\A_6$, and so $L_{\overline{\a}}\cong\A_4$. 
Under the natural action of $\A_6$ on $\{1,\ldots,6\}$, this subgroup fixes 5 and 6 pointwise. Thus the action of $L$ on the cosets of $L_{\overline{\a}}$ is permutation equivalent to the action of $\A_6$ on the ordered pairs $(i,j)$ with $i\neq j$, and we may identify $\overline{\a}$ with $(5,6)$. 
Let $I=\{1,2,3,4\}$ and set $\Delta_1=\{(5,i)\mid i\in I\}$, $\Delta_2=\{(i,6)\mid i\in I\}$, $\Delta_3=\{(6,i)\mid i\in I\}$ and $\Delta_4=\{(i,5)\mid i\in I\}$. The $L_{\overline{\a}}$-orbits have lengths 1, 1, 4, 4, 4, 4, 12, where $\Delta_1,\Delta_2,\Delta_3,\Delta_4$ are the four orbits of length 4.

We first exclude $\Delta_3$ and $\Delta_4$. Assume that $\Ga_N(\overline{\a})=\Delta_3$ and take $\overline{\b}=(6,i)\in \Delta_3$. The graph $\Ga_N$ is undirected, so $\overline{\a}$ is adjacent to $\overline{\b}$. Let $x\in L$ satisfy $\overline{\b}^x=\overline{\a}$. Then $6^x=5$ and $i^x=6$, whereas $\overline{\a}^x=(5^x,5)\in\Delta_4$. The element $x$ preserves adjacency, so $\overline{\a}^x$ is adjacent to $\overline{\a}$. Hence $\overline{\a}^x\in\Ga_N(\overline{\a})=\Delta_3$, a contradiction. Thus $\Ga_N(\overline{\a})\neq\Delta_3$. The same argument excludes $\Delta_4$.

Suppose that $\Ga_N(\overline{\a})=\Delta_1$. Then every vertex $(a,b)$ is adjacent precisely to the four vertices $(a,c)$ with $c\notin\{a,b\}$. Hence, for each $a\in\{1,\ldots,6\}$, the five vertices $(a,b)$ with $b\neq a$ induce a copy of $\K_5$, and there are no edges between two such sets with different first coordinates. Thus $\Ga_N\cong 6\K_5$, contrary to the assumption that $\Ga_N$ is connected. Similarly, $\Ga_N(\overline{\a})=\Delta_2$ gives $\Ga_N\cong 6\K_5$, with the six components determined by the second coordinate. Then $|\Ga_N(\overline{\a})|\neq 4$. The only remaining orbit length greater than 1 is 12, which is not a prime power. Hence the case in \cite[Proposition 4.3(d)]{LX2022} does not occur.

Next consider the case in \cite[Proposition 4.3(e)]{LX2022}, that is, $L\cong \A_6$, $T\cong \A_5$ and $L_{\overline{\a}}\leqslant M=(\Sy_3\wr\Sy_2)\cap \A_6\cong \ZZ_3^2:\ZZ_4$, where $|M|=36$. The factorization $L=TL_{\overline{\a}}$ gives $|L_{\overline{\a}}:T\cap L_{\overline{\a}}|=|L:T|=6$, and so $|L_{\overline{\a}}|=6|T\cap L_{\overline{\a}}|$. 
The intersection $T\cap L_{\overline{\a}}$ is nontrivial and $|L_{\overline{\a}}|$ divides 36. It follows that $|T\cap L_{\overline{\a}}|\in\{2,3,6\}$ and $|L_{\overline{\a}}|\in\{12,18,36\}$.
The Magma computation in Appendix \ref{app:A6-wreath} shows that $M$ has no subgroup of order 12, and that $|T\cap H|=6$ for every subgroup $H\leqslant M$ of order 18 and every subgroup $T\cong\A_5$ of $L$. Thus $|L_{\overline{\a}}|=36$, so $L_{\overline{\a}}=M$. The same computation shows that $L$ acts 2-transitively on $[L:M]$, with subdegrees 1 and 9. Hence $\Ga_N\cong\K_{10}$, contrary to the assumption that $\Ga_N$ is not complete. Thus \cite[Proposition 4.3(e)]{LX2022} is excluded.

For \cite[Proposition 4.3 (f)]{LX2022} and \cite[Table 4.2]{LX2022}, no row has both $T$ and $L$ nonabelian simple. Hence no further alternating case arises, and Table \ref{tab:AS2} is exhaustive.

Suppose next that $L$ is a sporadic simple group. \cite[Proposition 4.4 (a)]{LX2022}, applied to the factorization $L=L_{\overline{\a}}T$, gives $(L,T)=(\M_{12},\M_{11})$, with $L_{\overline{\a}}$ transitive of degree 12. The nontrivial intersection $T\cap L_{\overline{\a}}$ gives $|L_{\overline{\a}}|>|L:T|=12$, producing the first row of Table \ref{tab:AS3}. \cite[Proposition 4.4 (b)]{LX2022} gives $(L,T)=(\M_{24},\M_{23})$, where $L_{\overline{\a}}$ is transitive of degree $24$. Again $T\cap L_{\overline{\a}}\neq1$, and so $|L_{\overline{\a}}|>24$. This gives the second row of Table \ref{tab:AS3}.
For the case in \cite[Proposition 4.4 (c)]{LX2022}, both $L$ and $T$ are nonabelian simple. Hence rows 1, 4, 5, 8 and 10--13 of \cite[Table 4.3]{LX2022} are excluded because either $L$ or $T$ is not simple. It remains to consider the other rows of that table.
Row 2 of \cite[Table 4.3]{LX2022} gives $(L,T)=(\M_{11},\PSL_2(11))$, with $L_{\overline{\a}}\cong\M_9$, $\M_9.\ZZ_2$ or $\AGL_1(9)$. These are recorded in the third row of Table \ref{tab:AS3}. 
Row 3 of \cite[Table 4.3]{LX2022} gives $(L,T)=(\M_{12},\PSL_2(11))$, with $L_{\overline{\a}}\cong\M_9.\ZZ_2$ or $\M_9.\Sy_3$. The equality $|\M_{12}:\PSL_2(11)|=144=|\M_9.\ZZ_2|$ would force $T\cap L_{\overline{\a}}=1$, contrary to the present assumption. Hence only $L_{\overline{\a}}\cong\M_9.\Sy_3$ remains, giving the row 4 of Table \ref{tab:AS3}.
Rows 6 and 7 of \cite[Table 4.3]{LX2022} give $(L,T)=(\M_{23},\M_{22})$. Row 6 of \cite[Table 4.3]{LX2022} has $|L_{\overline{\a}}|=|L:T|=23$, again forcing $T\cap L_{\overline{\a}}=1$, so this row is excluded. 
In row 7 of \cite[Table 4.3]{LX2022}, $L_{\overline{\a}}\cong \ZZ_{23}:\ZZ_{11}$. The local socle has order $p^e$ with $e\geqslant 2$, and so $p^e\mid |L_{\overline{\a}}|=23\cdot11$, which is impossible. 
Row 9 of \cite[Table 4.3]{LX2022} gives $(L,T)=(\HS,\M_{22})$ with $L_{\overline{\a}}\cong5_+^{1+2}:\ZZ_8:\ZZ_2$, producing the fifth row of Table \ref{tab:AS3}. Thus the sporadic candidates are exactly those listed in Table \ref{tab:AS3}.

\begin{table}[htbp]
\centering
\setlength{\tabcolsep}{10pt}
\renewcommand{\arraystretch}{1.3}
\caption{Sporadic candidates with $L_{\overline{\a}}$ solvable and $e\geqslant 2$.}
\label{tab:AS3}
{\scriptsize
\begin{tabular}{|c|c|c|c|}
\hline
$L$ & $T$ & $L_{\overline{\a}}$ & remark\\
\hline
$\M_{12}$ & $\M_{11}$ & transitive of degree 12 & $|L_{\overline{\a}}|>12$\\
\hline
$\M_{24}$ & $\M_{23}$ & transitive of degree 24 & $|L_{\overline{\a}}|>24$\\
\hline
$\M_{11}$ & $\PSL_2(11)$ & $\M_9$, $\M_9.\ZZ_2$, $\AGL_1(9)$ & \\
\hline
$\M_{12}$ & $\PSL_2(11)$ & $\M_9.\Sy_3$ & \\
\hline
$\HS$ & $\M_{22}$ & $5_+^{1+2}:\ZZ_8:\ZZ_2$ & \\
\hline
\end{tabular}}
\end{table}

Suppose finally that $L$ is a classical simple group of Lie type. By \cite[Theorem 1.1 (d)]{LX2022}, the factorization $L=L_{\overline{\a}}T$ is listed in \cite[Tables 1.1 and 1.2]{LX2022}.
We first consider \cite[Table 1.1]{LX2022}. In row 1, the group $T$ contains a normal subgroup $\E_q^{n-1}:\SL_{n-1}(q)$, contradicting the simplicity of $T$. In row 2, $L=\PSL_4(q)$, $T=\PSp_4(q)$ and $L_{\overline{\a}}\leqslant U$, where $|U|=3q^3(q^3-1)/(4,q-1)$. For odd $q$, we have $|L:T|=2q^2(q^3-1)/(4,q-1)$, so $|U|/|L:T|=3q/2$ is not integer. Thus $q$ is even. Moreover, $q\neq 2$, since $\PSp_4(2)$ is not simple. Applying $|L:T|\mid |L_{\overline{\a}}|$ to rows 3--9 gives the remaining possibilities in Table \ref{tab:AS4}.

\begin{table}[htbp]
\centering
\setlength{\tabcolsep}{10pt}
\renewcommand{\arraystretch}{1.3}
\caption{Infinite classical candidates with $L_{\overline{\a}}$ solvable and $e\geqslant 2$.}
\label{tab:AS4}
{\scriptsize
\begin{tabular}{|c|c|c|c|}
\hline
$L$ & $T$ & $L_{\overline{\a}}\leqslant$ & remark\\
\hline
$\PSL_4(q)$ & $\PSp_4(q)$ & $q^3:\dfrac{q^3-1}{(4,q-1)}.3$ & $q\geqslant 4$ even\\
\hline
$\PSp_{2m}(q)$ & $\Omega^-_{2m}(q)$ & $q^{m(m+1)/2}:(q^m-1).m$ & $m>2$, $q$ even\\
\hline
$\PSp_4(q)$ & $\Sp_2(q^2)$ & $q^3:(q^2-1).2$ & $q\geqslant 4$ even\\
\hline
$\PSp_4(q)$ & $\PSp_2(q^2)$ & $q^{1+2}:\dfrac{q^2-1}{2}.2$ & $q$ odd\\
\hline
$\PSU_{2m}(q)$ & $\SU_{2m-1}(q)$ & $q^{m^2}:\dfrac{q^{2m}-1}{(q+1)(2m,q+1)}.m$ & $m>2$, $(q+1)\mid m$\\
\hline
$\Omega_{2m+1}(q)$ & $\Omega^-_{2m}(q)$ & $(q^{m(m-1)/2}.q^m):\dfrac{q^m-1}{2}.m$ & $m>3$ even, $q$ odd\\
\hline
$\POmega^+_{2m}(q)$ & $\Omega_{2m-1}(q)$ & $q^{m(m-1)/2}:\dfrac{q^m-1}{(4,q^m-1)}.m$ & $q$ even or $m>5$ even\\
\hline
$\POmega^+_8(q)$ & $\Omega_7(q)$ & $q^6:\dfrac{q^4-1}{(4,q^4-1)}.4$ & \\
\hline
\end{tabular}}
\end{table}

We next consider \cite[Table 1.2]{LX2022}. The group $L_{\overline{\a}}^{\Ga_N(\overline{\a})}$ is a quotient of $L_{\overline{\a}}$, and $\soc(L_{\overline{\a}}^{\Ga_N(\overline{\a})})\cong \ZZ_p^e$ with $e\geqslant 2$. Hence $p^e\mid |L_{\overline{\a}}|$, which excludes rows 1, 3 and 4 of \cite[Table 1.2]{LX2022}. Moreover, the factorization $L=TL_{\overline{\a}}$ gives $|L:T|=|L_{\overline{\a}}:T\cap L_{\overline{\a}}|$, so $|L:T|\mid |L_{\overline{\a}}|$. This divisibility excludes rows 2, 6--10 and 17.
Row 5 is also impossible. Here $(L,T)=(\PSL_2(59),\A_5)$ and $|L:T|=1711=59\cdot 29$, while $L_{\overline{\a}}\leqslant 59:29$. Thus $|L_{\overline{\a}}|=1711$ and $T\cap L_{\overline{\a}}=1$, contrary to the present assumption.
In row 11, the simplicity of $T$ leaves $T\cong\A_5$ or $\A_6$. The first possibility is excluded by $|\PSp_4(3):\A_5|=432\nmid 648$, whereas $|\PSp_4(3):\A_6|=72\mid 648$. Thus row 11 contributes only $(L,T)=(\PSp_4(3),\A_6)$.
For rows 12--16 and 18--27, retaining only the possibilities for which $T$ is nonabelian simple and $|L:T|\mid |L_{\overline{\a}}|$ gives precisely the entries in Table \ref{tab:AS5}. This completes the classical case and claim 1.

\begin{table}[htbp]
\centering
\setlength{\tabcolsep}{4pt}
\renewcommand{\arraystretch}{1.3}
\caption{Finite classical possibilities with $L_{\overline{\a}}$ solvable and $e\geqslant 2$.}
\label{tab:AS5}
{\scriptsize
\begin{tabular}{|c|c|c|c|c|c|}
\hline
$L$ & $T$ & $L_{\overline{\a}}$ & $L$ & $T$ & $L_{\overline{\a}}$\\
\hline
$\PSp_4(3)$ & $\A_6$ & $3_+^{1+2}:2.\A_4$ & $\PSp_4(5)$ & $\PSL_2(25)$ & $5_+^{1+2}:4.\A_4$\\
\hline
$\PSp_4(7)$ & $\PSL_2(49)$ & $7_+^{1+2}:6.\Sy_4$ & $\PSp_4(11)$ & $\PSL_2(121)$ & $11_+^{1+2}:10.\A_4$\\
\hline
$\PSp_4(23)$ & $\PSL_2(529)$ & $23_+^{1+2}:22.\Sy_4$ & $\Sp_6(2)$ & $\A_8$ & $3_+^{1+2}:2.\Sy_4$\\
\hline
$\PSU_3(3)$ & $\PSL_2(7)$ & $3_+^{1+2}:8$ & $\PSU_3(5)$ & $\A_7$ & $5_+^{1+2}:8$\\
\hline
$\PSU_4(3)$ & $\PSL_3(4)$ & $3^4:\D_{10}$, $3^4:\Sy_4$ & $\Omega_7(3)$ & $\G_2(3)$ & $3^5:2^4.\AGL_1(5)$\\
\hline
$\PSU_4(3)$ & $\PSL_3(4)$ & $3^4:3^2:4$, $3_+^{1+4}.2.\Sy_4$ & $\Omega_7(3)$ & $\Sp_6(2)$ & $3^{3+3}:13:3$\\
\hline
$\Omega_9(3)$ & $\Omega^-_8(3)$ & $3^{6+4}:2^{1+4}.\AGL_1(5)$ & $\Omega^+_8(2)$ & $\Sp_6(2)$ & $2^2:15.4<\A_9$\\
\hline
$\Omega^+_8(2)$ & $\A_9$ & $2^6:15.4$ & $\POmega^+_8(3)$ & $\Omega_7(3)$ & $3^6:2^4.\AGL_1(5)$\\
\hline
$\POmega^+_8(3)$ & $\Omega^+_8(2)$ & $3^6:(3^3:13:3)$ & $\POmega^+_8(3)$ & $\Omega^+_8(2)$ & $3^{3+6}:13.3$\\
\hline
\end{tabular}}
\end{table}

\smallskip
\noindent{\textbf{Claim 2.} If $L_{\overline{\a}}$ is nonsolvable, then $\Ga_N$ is $(L,2)$-arc-transitive.}

Let $H=L_{\overline{\a}}^{\Ga_N(\overline{\a})}$ and $Y=\overline{X}_{\overline{\a}}^{\Ga_N(\overline{\a})}$. Since $\overline{X}$ is almost simple with socle $L$, we have $L\unlhd\overline{X}$ and $\overline{X}/L\leqslant\Out(L)$. Hence $L_{\overline{\a}}\unlhd\overline{X}_{\overline{\a}}$, and so $H\unlhd Y$. 
Moreover, $Y/H$ is a quotient of $\overline{X}_{\overline{\a}}/L_{\overline{\a}}$, giving $|Y:H|\mid |\overline{X}_{\overline{\a}}:L_{\overline{\a}}|$. Both $L$ and $\overline{X}$ are vertex-transitive on $V(\Ga_N)$, so $|L:L_{\overline{\a}}|=|V(\Ga_N)|=|\overline{X}:\overline{X}_{\overline{\a}}|$. Hence $|\overline{X}_{\overline{\a}}:L_{\overline{\a}}|=|\overline{X}:L|$, and so $|Y:H|\mid |\Out(L)|$.

Recall that $\Ga_N$ is connected and $(\overline{X},2)$-arc-transitive, while $L\unlhd\overline{X}$ and $L_{\overline{\a}}\neq1$. Applying \cite[Lemmas 4.1 and 4.2]{L2025} to $\overline{X}$ and its normal subgroup $L$, together with the nonsolvability of $L_{\overline{\a}}$, shows that $L_{\overline{\a}}$ acts transitively on $\Ga_N(\overline{\a})$ and that one of the following holds:
\begin{enumerate}[label=(\alph*),leftmargin=2em]
\item $H$ is 2-transitive on $\Ga_N(\overline{\a})$;
\item $|\Ga_N(\overline{\a})|=28$, $H\cong\PSL_2(8)$ and $Y\cong \PGammaL_2(8)$;
\item $|\Ga_N(\overline{\a})|=p^2$ and $\ZZ_p^2:\SL_2(5)\unlhd H\unlhd Y\leqslant \ZZ_p^2:(\ZZ_{p-1}.\PSL_2(5))$, where $p\in\{19,29,59\}$.
\end{enumerate}

Since $L$ is vertex-transitive on $V(\Ga_N)$ and $L_{\overline{\a}}$ is transitive on $\Ga_N(\overline{\a})$, the graph $\Ga_N$ is $L$-arc-transitive. We next assume that (a) fails, that is, $H$ is not 2-transitive on $\Ga_N(\overline{\a})$. We show in two steps that, when $L_{\overline{\a}}$ is nonsolvable, Cases (b) and (c) do not occur.

\vspace{0.5em}
\emph{Step 1. We prove that Case (b) does not occur.} Assume that (b) holds, that is, $|\Ga_N(\overline{\a})|=28$, $H\cong\PSL_2(8)$ and $Y\cong\PGammaL_2(8)$. Then $|Y:H|=3$, and so $3\mid |\Out(L)|$. 
Take $\overline{\b}\in\Ga_N(\overline{\a})$. By \cite[Corollary 2.3]{G1973}, $L_{\overline{\a}\overline{\b}}^{[1]}=L_{\overline{\a}}^{[1]}\cap L_{\overline{\b}}^{[1]}$ has prime-power order. 
Moreover, $L_{\overline{\a}}^{[1]}/L_{\overline{\a}\overline{\b}}^{[1]}\cong (L_{\overline{\a}}^{[1]})^{\Ga_N(\overline{\b})}\leqslant (L_{\overline{\b}}^{\Ga_N(\overline{\b})})_{\overline{\a}}\cong (L_{\overline{\a}}^{\Ga_N(\overline{\a})})_{\overline{\b}}=H_{\overline{\b}}$.
Since $H\cong \PSL_2(8)$ is transitive of degree 28, we have $|H_{\overline{\b}}|=|\PSL_2(8)|/28=18$, and in this action $H_{\overline{\b}}\cong \ZZ_9:\ZZ_2$. Thus $L_{\overline{\a}}^{[1]}/L_{\overline{\a}\overline{\b}}^{[1]}$ is solvable. 
The subgroup $L_{\overline{\a}\overline{\b}}^{[1]}$ has prime-power order and is also solvable, so $L_{\overline{\a}}^{[1]}$ is solvable.
Finally, $L_{\overline{\a}}/L_{\overline{\a}}^{[1]}\cong H\cong\PSL_2(8)$, and so $\PSL_2(8)$ is the unique nonabelian composition factor of $L_{\overline{\a}}^{(\infty)}$.

We first show that $|L_{\overline{\a}\overline{\b}}^{[1]}|$ is odd. Suppose otherwise that $L_{\overline{\a}\overline{\b}}^{[1]}$ is a nontrivial 2-group. Since $L_{\overline{\a}}^{[1]}\unlhd L_{\overline{\a}}$ and $L_{\overline{\a}}^{[1]}\leqslant L_{\overline{\a}\overline{\b}}$, we have $L_{\overline{\a}}^{[1]}\unlhd L_{\overline{\a}\overline{\b}}$. Thus $(L_{\overline{\a}}^{[1]})^{\Ga_N(\overline{\b})}$ is normal in $L_{\overline{\a}\overline{\b}}^{\Ga_N(\overline{\b})}\cong\ZZ_9:\ZZ_2$, and so  $L_{\overline{\a}}^{[1]}/L_{\overline{\a}\overline{\b}}^{[1]}\unlhd\ZZ_9:\ZZ_2$. Since $\O_2(\ZZ_9:\ZZ_2)=1$, we obtain that $\O_2(L_{\overline{\a}}^{[1]})\leqslant L_{\overline{\a}\overline{\b}}^{[1]}$. 
On the other hand, $L_{\overline{\a}\overline{\b}}^{[1]}$ is a normal 2-subgroup of $L_{\overline{\a}}^{[1]}$, so $L_{\overline{\a}\overline{\b}}^{[1]}\leqslant\O_2(L_{\overline{\a}}^{[1]})$. Thus $L_{\overline{\a}\overline{\b}}^{[1]}=\O_2(L_{\overline{\a}}^{[1]})$. Since $\O_2(L_{\overline{\a}}^{[1]})$ is characteristic in $L_{\overline{\a}}^{[1]}$ and $L_{\overline{\a}}^{[1]}\unlhd L_{\overline{\a}}$, we get that $L_{\overline{\a}\overline{\b}}^{[1]}\unlhd L_{\overline{\a}}$.

The subgroup $L_{\overline{\a}\overline{\b}}^{[1]}$ has odd order, and $L_{\overline{\a}}^{[1]}/L_{\overline{\a}\overline{\b}}^{[1]}\lesssim H_{\overline{\b}}\cong \ZZ_9:\ZZ_2$. Hence $v_2(|L_{\overline{\a}}^{[1]}|)\leqslant 1$. 
Combining this with $L_{\overline{\a}}/L_{\overline{\a}}^{[1]}\cong \PSL_2(8)$ and $v_2(|\PSL_2(8)|)=3$ gives $v_2(|L_{\overline{\a}}|)\leqslant 4$.

In what follows, we consider the possibilities for $L$ given by the classification of finite nonabelian simple groups. By \cite{Atlas}, $|\Out(\A_n)|=2$ for $n\geqslant 5$ and $n\neq 6$, while $|\Out(\A_6)|=4$, and $|\Out(L)|\leqslant2$ for every sporadic simple group $L$. By the preceding result, $3\mid |\Out(L)|$, which excludes alternating and sporadic groups $L$.

Suppose that $L$ is an exceptional simple group of Lie type. By \cite[Theorem 1]{HLS1987}, up to interchanging the two factors, the relevant factorizations are of the following types:
\begin{enumerate}[label=(\roman*),leftmargin=2em]
\item $L\cong\G_2(q)$ with nonsolvable factors $\SL_3(q)$ and $\SU_3(q)$, where $q=3^f$;
\item $L\cong\G_2(q)$ with nonsolvable factors ${}^2\G_2(q)$ and $\SL_3(q)$, where $q=3^{2f+1}$;
\item $L\cong\G_2(4)$ with factors $\J_2$ and $\SU_3(4)$;
\item $L\cong\F_4(q)$ with factors $\Sp_8(q)$ and ${}^3D_4(q)$, where $q=2^f$.
\end{enumerate}
Cases (i), (iii) and (iv) are excluded by the preceding conclusion that $\PSL_2(8)$ is the unique nonabelian composition factor of $L_{\overline{\a}}^{(\infty)}$. In case (ii), this can occur only for $q=3$, since ${}^2\G_2(3)'\cong\PSL_2(8)$; however, $|\Out(\G_2(3))|=2$, contradicting $3\mid |\Out(L)|$. Thus $L$ is not an exceptional simple group of Lie type.

Suppose finally that $L$ is a classical simple group of Lie type. Apply \cite[Theorem 1.4]{LWX2022} to the factorization $L=TL_{\overline{\a}}$. The simplicity of $T$ gives $T^{(\infty)}=T$, while $\PSL_2(8)$ is the unique nonabelian composition factor of $L_{\overline{\a}}^{(\infty)}$.
Consider first \cite[Table 3.1]{LWX2022}. In Row 1, the factor $\SL_2(8)$ occurs only for $(n,q,a,b)=(6,2,2,3)$, while the other nonsolvable factor is $\ZZ_2^5:\SL_5(2)$, which is not simple and hence cannot equal $T^{(\infty)}=T$. No other row contains a factor with unique nonabelian composition factor $\PSL_2(8)$. Hence \cite[Table 3.1]{LWX2022} gives no possibility, and the same holds for \cite[Table 4.1]{LWX2022}. In \cite[Table 5.1]{LWX2022}, Rows 1, 2 and 5 yield, $(L,T)=(\PSp_4(8),\PSL_2(64))$, $(\Sp_6(2),\A_8)$ and $(\Sp_{12}(2),\Omega^-_{12}(2))$, respectively, while the remaining rows give no further possibility. In \cite[Table 5.2]{LWX2022}, Row 6 contains $\SL_2(8)$, but the other nonsolvable factor is $\ZZ_2^4:\A_5$, whose unique nonabelian composition factor is $\A_5$. Thus \cite[Table 5.2]{LWX2022} contributes nothing further.
In \cite[Table 6.1]{LWX2022}, Row 3 with $f=1$ gives $(L,T)=(\Omega_7(3),\Omega_6^+(3))$, using ${}^2\G_2(3)'\cong\PSL_2(8)$. The remaining rows of \cite[Tables 6.1 and 7.1]{LWX2022} do not contain the required nonabelian composition factor. Finally, Row 1 of \cite[Table 8.1]{LWX2022}, with $(m,q,a,b)=(6,2,2,3)$, gives $(L,T)=(\POmega^+_{12}(2),\Sp_{10}(2))$, and no further pair arises from the remaining rows.
By the preceding analysis, $(L,T)$ is one of $(\PSp_4(8),\PSL_2(64))$, $(\Sp_6(2),\A_8)$, $(\Sp_{12}(2),\Omega^-_{12}(2))$, $(\Omega_7(3),\Omega_6^+(3))$ and $(\POmega^+_{12}(2),\Sp_{10}(2))$. For $(L,T)=(\PSp_4(8),\PSL_2(64))$, we have $|L:T|=8^2(8^2-1)$ and so $v_2(|L:T|)=6$. 
The factorization $L=TL_{\overline{\a}}$ gives $|L:T|=|L_{\overline{\a}}:T\cap L_{\overline{\a}}|$, so $|L:T|\mid |L_{\overline{\a}}|$, contradicting $v_2(|L_{\overline{\a}}|)\leqslant 4$. For $L\in\{\Sp_6(2),\Sp_{12}(2),\Omega_7(3),\POmega^+_{12}(2)\}$, \cite{Atlas} gives $|\Out(L)|\leqslant 2$, contradicting the preceding conclusion that $3\mid |\Out(L)|$. This proves Step 1.

\vspace{0.5em}

\emph{Step 2. We show that case (c) does not occur.} Assume that (c) holds, that is, $|\Ga_N(\overline{\a})|=p^2$ and $\ZZ_p^2:\SL_2(5)\unlhd H\unlhd Y\leqslant \ZZ_p^2:(\ZZ_{p-1}.\PSL_2(5))$, where $p\in\{19,29,59\}$.  Let $E=\soc(H)=\soc(Y)\cong\ZZ_p^2$, and regard $E$ as a 2-dimensional vector space over $\mathbb{F}_p$. The subgroup $\SL_2(5)$ acts faithfully on $E$, so $\SL_2(5)\leqslant \GL_2(p)$. Since $\SL_2(5)'=\SL_2(5)$ and $\GL_2(p)/\SL_2(p)$ is abelian, we have $\SL_2(5)\leqslant \SL_2(p)$. 
For $0\neq v\in E$, let $w\in E$ such that $(v,w)$ is a basis of $E$. An element of $\SL_2(p)$ fixing $v$ sends $w$ to $w+av$ for some $a\in\mathbb{F}_p$, and so $|\SL_2(p)_v|=p$. Thus $|\SL_2(5)_v|$ divides both $p$ and $|\SL_2(5)|=120$. Since $p\nmid 120$, we get that $\SL_2(5)_v=1$. Hence every $\SL_2(5)$-orbit on $E\setminus\{0\}$ has length 120.

Let $H_0$ and $Y_0$ be the stabilizers of 0 in $H$ and $Y$, respectively, and let $s$ be the number of $H_0$-orbits on $E\setminus\{0\}$. The group $H$ is not 2-transitive on $E$, so $H_0$ is not transitive on $E\setminus\{0\}$ and $s>1$. 
Recall that $\Ga_N$ is $(\overline{X},2)$-arc-transitive. Hence $\overline{X}_{\overline{\a}}$ induces a 2-transitive permutation group on $\Ga_N(\overline{\a})$, and so $Y$ is 2-transitive on $E$. Thus $Y_0$ is transitive on $E\setminus\{0\}$. 
The normality $H\unlhd Y$ gives $H_0\unlhd Y_0$, so $Y_0$ permutes the $H_0$-orbits on $E\setminus\{0\}$ transitively. Hence these $s$ orbits have equal length. 
Moreover, $(E:\SL_2(5))_0\cong \SL_2(5)$ is normal in $H_0$, and every $H_0$-orbit is a union of $\SL_2(5)$-orbits of length $120$. It follows that $120s\mid p^2-1$.

The regularity of $E$ yields $H=E:H_0$ and $Y=E:Y_0$, whence $|Y:H|=|Y_0:H_0|$. Since $H_0\unlhd Y_0$ and $Y_0$ permutes the $s$ orbits transitively, $s\mid |Y:H|$. For $p=19,29,59$, the values of $(p^2-1)/120$ are 3, 7 and 29, respectively. The inequalities $s>1$ and $120s\mid p^2-1$ now imply $s=(p^2-1)/120$. The earlier divisibility $|Y:H|\mid |\Out(L)|$ gives $s\mid |\Out(L)|$.

By \cite[Lemma 4.2]{L2025}, every nonabelian composition factor of $L_{\overline{\a}}$ is a nonabelian composition factor of either $H$ or $H_{\overline{\b}}$ for some $\overline{\b}\in\Ga_N(\overline{\a})$. Note that $H\leqslant E:(\ZZ_{p-1}.\PSL_2(5))$. Since $E:\ZZ_{p-1}$ is solvable and every proper subgroup of $\PSL_2(5)\cong\A_5$ is solvable, every nonabelian composition factor of both $H$ and $H_{\overline{\b}}$ is isomorphic to $\A_5$. Hence every nonabelian composition factor of $L_{\overline{\a}}^{(\infty)}$ is isomorphic to $\A_5$.

Since $\SL_2(5)'=\SL_2(5)$ and $[\ZZ_p^2,\SL_2(5)]=\ZZ_p^2$, we have $(\ZZ_p^2:\SL_2(5))'=\ZZ_p^2:\SL_2(5)$. Moreover, $\ZZ_p^2:\SL_2(5)\unlhd H$ and $H/(\ZZ_p^2:\SL_2(5))$ is solvable. Hence $H^{(\infty)}=\ZZ_p^2:\SL_2(5)$, and so $p^2\mid |H^{(\infty)}|$.
The equality $H=L_{\overline{\a}}/L_{\overline{\a}}^{[1]}$ gives
$H^{(\infty)}=L_{\overline{\a}}^{(\infty)}L_{\overline{\a}}^{[1]}/L_{\overline{\a}}^{[1]}\cong L_{\overline{\a}}^{(\infty)}/(L_{\overline{\a}}^{(\infty)}\cap L_{\overline{\a}}^{[1]})$. Hence $|H^{(\infty)}|\mid |L_{\overline{\a}}^{(\infty)}|$. Together with $p^2\mid |H^{(\infty)}|$, this yields $p^2\mid |L_{\overline{\a}}^{(\infty)}|$.

Next we consider the possibilities for $L$ given by the classification of finite nonabelian simple groups.
By \cite{Atlas}, $|\Out(\A_n)|=2$ for $n\geqslant 5$ and $n\neq 6$, while $|\Out(\A_6)|=4$, and every sporadic simple group has outer automorphism group of order at most 2. The divisibility $s\mid |\Out(L)|$, where $s\in\{3,7,29\}$, excludes alternating and sporadic groups $L$. 
Suppose that $L$ is an exceptional simple group of Lie type. By \cite[Theorem 1]{HLS1987}, up to interchanging the two factors, the relevant factorizations $L=TL_{\overline{\a}}$ are those listed there. None of them has a factor whose nonabelian composition factors are all isomorphic to $\A_5$ and whose order is divisible by $p^2$ for $p\in\{19,29,59\}$. Thus $L$ is not an exceptional simple group of Lie type.

Suppose finally that $L$ is a classical simple group. By \cite[Theorem 1.4]{LWX2022}, after interchanging the two factors if necessary, $(L,T,L_{\overline{\a}}^{(\infty)})$ occurs in \cite[Tables 3.1, 4.1, 5.1, 5.2, 6.1, 7.1 and 8.1]{LWX2022}. Recall that every nonabelian composition factor of $L_{\overline{\a}}^{(\infty)}$ is isomorphic to $\A_5$ and $p^2\mid |L_{\overline{\a}}^{(\infty)}|$, where $p\in\{19,29,59\}$.
In \cite[Table 3.1]{LWX2022}, only Rows 1, 7 and 10 admit a nonsolvable factor whose nonabelian composition factors are all isomorphic to $\A_5$. However, the orders of these factors are not divisible by $p^2$ for $p\in\{19,29,59\}$.  For \cite[Table 4.1]{LWX2022}, the only such possibility occurs in Row 3, where the relevant factor is $3^4:\A_5$, whose order is not divisible by $p^2$. The same argument excludes all possibilities in \cite[Tables 5.1, 6.1, 7.1 and 8.1]{LWX2022}.

It remains to consider \cite[Table 5.2]{LWX2022}. In Row 6 of \cite[Table 5.2]{LWX2022}, the factor $\ZZ_2^4:\A_5$ has order not divisible by $p^2$, while in Row 9 of \cite[Table 5.2]{LWX2022} the factor $3^{1+4}:2^{1+4}.\A_5$ also has order not divisible by $p^2$. The remaining rows are excluded by the two conditions on the nonabelian composition factors and the order of $L_{\overline{\a}}^{(\infty)}$, except possibly Rows 1--5. In these five rows, the condition $p^2\mid |L_{\overline{\a}}^{(\infty)}|$, where $p\in\{19,29,59\}$, leaves only Rows 3--5. Hence $L\cong\PSp_4(p)$, $T\cong\PSL_2(p^2)$ and $L_{\overline{\a}}^{(\infty)}\cong p^{1+2}:\SL_2(5)$.
For $p\in\{19,29,59\}$, we have $|\Out(\PSp_4(p))|=2$, contradicting $(p^2-1)/120\mid |\Out(L)|$ with $(p^2-1)/120=3$, 7 or 29. Thus case (c) does not occur. This proves Step 2.

The non-2-transitive alternatives (b) and (c) have both been eliminated. Thus $H=L_{\overline{\a}}^{\Ga_N(\overline{\a})}$ is 2-transitive on $\Ga_N(\overline{\a})$. Since $L$ is vertex-transitive, $\Ga_N$ is $(L,2)$-arc-transitive, proving Claim 2. By Claims 1 and 2, together with the case $T=L$ and the case $N\leqslant G$, this proves the lemma.  \qed

\vspace{0.5em}

\begin{lemma}\label{lem:diagonal2}
The group $\overline{X}$ is not of simple diagonal type.
\end{lemma}

\pf\, By Lemma \ref{lem:normal-quotient}, the graph $\Ga_N$ is connected $(\overline{X},2)$-arc-transitive, and $\overline{X}$ is quasiprimitive on $V(\Ga_N)$. By \cite[Theorem 2]{P1993}, the group $\overline{X}$ is of affine, almost simple, product action or twisted wreath type. Thus $\overline{X}$ is not of simple diagonal type. \qed

\vspace{0.5em}

\begin{lemma}\label{lem:PA2}
Suppose that $\overline{X}$ is of product action type. Let $\soc(\overline{X})=Q_1\times\cdots\times Q_r\cong Q^r$, where $r\geqslant 2$ and each $Q_i\cong Q$ is nonabelian simple. If $\overline{G}\ntrianglelefteq \overline{X}$, then $(Q,T,\ell)$ is listed in Table \ref{tab:equal-prime-r}. If $N\leqslant G$, then $(Q,T,\ell)$ is listed in Table \ref{tab:exact-pa}.
\end{lemma}

\pf\, By Proposition \ref{subgroup}, $\overline{G}\leqslant L:=\soc(\overline{X})$. Since $\overline{X}$ is of product action type, \cite[Section 2]{P1993} gives $\overline{X}\leqslant Y\wr\Sy_r$ in product action, where $Y$ is quasiprimitive and $\soc(Y)\cong Q$. The group $Y$ is either almost simple or of simple diagonal type. By Lemma \ref{lem:normal-quotient}, $\Ga_N$ is connected $(\overline{X},2)$-arc-transitive. The simple diagonal possibility is excluded by \cite[Theorem 2]{P1993}, leaving $Y$ almost simple.
For $1\leqslant i\leqslant r$, let $\pi_i:L\to Q_i$ be the natural projection and let $R_i=\pi_i(L_{\overline{\a}})$. By \cite[Section 2]{P1993}, $1<R_i<Q_i$ for every $i$, the groups $R_1,\ldots,R_r$ are pairwise isomorphic, and $\pi_i(L_{\overline{\a}})=R_i$. Thus $L_{\overline{\a}}$ is a subdirect subgroup of $R_1\times\cdots\times R_r$.
Both $L$ and $\overline{G}$ are transitive on $V(\Ga_N)$, and so $L=\overline{G}L_{\overline{\a}}$. Hence $|L|=|\overline{G}||L_{\overline{\a}}|/|\overline{G}\cap L_{\overline{\a}}|$, and therefore $|L:\overline{G}|=|L_{\overline{\a}}:\overline{G}\cap L_{\overline{\a}}|$. In particular, $|L:\overline{G}|\mid |L_{\overline{\a}}|$.

Assume first that $\overline{G}\ntrianglelefteq\overline{X}$. Lemma \ref{lem:normal-quotient} implies that $\Ga_N$ is non-bipartite. Choose an edge $\{\overline{\a},\overline{\b}\}$ of $\Ga_N$. By \cite[Lemma 2.1]{Lu2026}, $\overline{X}=\l\overline{X}_{\overline{\a}},\overline{X}_{\overline{\b}}\r$, so \cite[Theorem 1.1]{Lu2026} applies with socle $L$. 
We consider separately the two possibilities according as all the restrictions $\pi_i|_{L_{\overline{\a}}}$ are injective or at least one of them is not injective. In the latter case, \cite[Theorem 1.1(2)]{Lu2026} also determines the structure of $L_{\overline{\a}}$.

\smallskip
\noindent\textbf{Case 1.} Every restriction $\pi_i|_{L_{\overline{\a}}}$ is injective.

For every $i$, the map $\pi_i|_{L_{\overline{\a}}}:L_{\overline{\a}}\to R_i$ is an isomorphism, so $|L_{\overline{\a}}|=|R_i|<|Q|$. We first prove that $\ell\geqslant 2$ and $T\not\cong Q$. Assume that $\ell=1$. Then $\overline{G}\neq L$, and Lemma \ref{lem:framework}(2)(d)(ii) gives $r=2$. The nontrivial group $\overline{G}\cong T$ has a nontrivial projection $\pi_i(\overline{G})$ for some $i$. 
The kernel of $\pi_i|_{\overline{G}}$ is normal in the simple group $\overline{G}$, so $\pi_i(\overline{G})\cong T\leqslant Q_i$ and $|T|\leqslant |Q|$. On the other hand, $L=\overline{G}L_{\overline{\a}}$ implies $|L|\leqslant |\overline{G}||L_{\overline{\a}}|$. It follows that $|Q|^2=|L|\leqslant |\overline{G}||L_{\overline{\a}}|\leqslant |Q||L_{\overline{\a}}|<|Q|^2$, a contradiction. Thus $\ell\geqslant 2$.

Suppose next that $T\cong Q$. The divisibility $|\overline{G}|\mid |L|$ gives $\ell\leqslant r$. If $\ell=r$, then $|\overline{G}|=|L|$ and $\overline{G}\leqslant L$, forcing $\overline{G}=L\unlhd\overline{X}$, contrary to our assumption. Thus $\ell<r$ and so $|L:\overline{G}|=|Q|^{r-\ell}\geqslant |Q|$. This is impossible because $|L:\overline{G}|\mid |L_{\overline{\a}}|$ whereas $|L_{\overline{\a}}|<|Q|$. Then $T\not\cong Q$.

Assume that $Q_i\leqslant \overline{G}$ for some $i$. The inclusions $Q_i\unlhd L$ and $\overline{G}\leqslant L$ show that $Q_i\unlhd \overline{G}\cong T^\ell$. By Proposition \ref{obs}, $Q_i\cong T$, contrary to $T\not\cong Q$.
Hence $\overline{G}$ contains no minimal normal subgroup of $L$.
We may therefore apply \cite[Theorem 1.3 (3)]{DS2020}. Either $\ell=r$ and, after relabeling, $T_i<Q_i$ for every $i$, or the factors $T_i$ are pairwise disjoint strips with some strip involving two of the groups $Q_j$, or $L\cong (\A_n)^r$, $|V(\Ga_N)|=n^r$ and $L_{\overline{\a}}\cong (\A_{n-1})^r$ for some $n\geqslant 10$.

We first compare the prime divisors of $Q$ and $T$. Let $s$ be a prime divisor of $|Q|$. The equality $|L_{\overline{\a}}|=|R_i|$ and the inclusion $R_i<Q_i$ imply $v_s(|L_{\overline{\a}}|)\leqslant v_s(|Q|)$. Hence $v_s(|V(\Ga_N)|)=rv_s(|Q|)-v_s(|L_{\overline{\a}}|)\geqslant (r-1)v_s(|Q|)>0$, so $s\mid |V(\Ga_N)|$. 
The transitive action of $\overline{G}$ implies that $|V(\Ga_N)|$ divides $|\overline{G}|=|T|^\ell$, and so $s\mid |T|$. Thus $\Pi(|Q|)\subseteq \Pi(|T|)$. The reverse inclusion follows from $\overline{G}\leqslant Q^r$, giving $\Pi(|Q|)=\Pi(|T|)$.

Lemma \ref{lem:strip-pair-exclusion} (i) excludes the strip case. In the alternating case, $|L_{\overline{\a}}|=|\A_{n-1}|^r\geqslant |\A_{n-1}|^2> |\A_n|=|Q|$, contrary to $|L_{\overline{\a}}|< |Q|$. Thus $\ell=r$ and, after relabeling, $T_i<Q_i$ for every $i$. Projecting $L=\overline{G}L_{\overline{\a}}$ onto $Q_i$ gives $Q_i=T_iR_i$. Here $T_i\cong T$, $R_i\cong L_{\overline{\a}}<Q_i\cong Q$, $r=\ell\geqslant 2$ and $\Pi(|Q|)=\Pi(|T|)$. Lemma \ref{lem:equal-prime-divisors} now places $(Q,T,r)$ in Table \ref{tab:equal-prime-r}. This proves the first assertion in Case 1.

\smallskip
\noindent{\textbf{Case 2.} Some restriction $\pi_i|_{L_{\overline{\a}}}$ is not injective.}

By \cite[Theorem 1.1(2)]{Lu2026}, $L_{\overline{\a}}\cong (\ZZ_p^t\times\ZZ_{m_1}).\ZZ_m$, where $|\Ga_N(\overline{\a})|=p^t$, $m_1\mid m$ and $m\mid p^e-1$ for some proper divisor $e$ of $t$. In particular, $L_{\overline{\a}}$ and every $R_i$ are solvable.
The connected $(\overline{X},2)$-arc-transitive graph $\Ga_N$ is basic with respect to $\overline{X}$ because $\overline{X}$ is quasiprimitive. It has valency at least 3 and, being non-bipartite, is not complete bipartite. Since $L=\soc(\overline{X})=Q_1\times\cdots\times Q_r$ with $r\geqslant 2$, \cite[Theorem 4.2]{LS2023} implies that every $Q_i$ is semiregular on $V(\Ga_N)$. In particular, $Q_i\cap L_{\overline{\a}}=1$ for each $i$. 
If $r=2$, then $\ker(\pi_1|_{L_{\overline{\a}}})=L_{\overline{\a}}\cap Q_2=1$ and $\ker(\pi_2|_{L_{\overline{\a}}})=L_{\overline{\a}}\cap Q_1=1$, contradicting the assumption of Case 2. Thus $r\geqslant 3$. Lemma \ref{lem:framework} (2)(d)(ii) also excludes $\ell=1$, so $\ell\geqslant 2$.

We next prove that $T\not\cong Q$. Assume the contrary. Then $\ell<r$, because $\ell=r$ would force $\overline{G}=L$. 
Let $\overline{G}=T_1\times\cdots\times T_\ell$. For $Q_i\in \mathcal{P}(T_j)=\{Q_i\mid \pi_i(T_j)\neq 1\}$, simplicity gives $\pi_i(T_j)\cong T_j\cong Q_i$, whence $\pi_i(T_j)=Q_i$. Distinct factors $T_j$ commute. 
No coordinate $Q_i$ can therefore lie in the supports of two distinct factors, and the sets $\mathcal{P}(T_1),\ldots,\mathcal{P}(T_\ell)$ are pairwise disjoint.  Moreover, their union is $\{Q_1,\ldots,Q_r\}$; otherwise $\pi_i(\overline{G})=1$ for some $i$, and projecting $L=\overline{G}L_{\overline{\a}}$ onto $Q_i$ would give $Q_i=R_i$, contrary to $R_i<Q_i$.

The inequality $r>\ell$ yields a factor $T_j$ whose support contains at least two coordinates. Relabel so that $Q_1,Q_2\in\mathcal{P}(T_j)$. 
The projection of $T_j$ onto $Q_1\times Q_2$ is a diagonal subgroup $\{(x,x^\varphi)\mid x\in Q_1\}$ for some isomorphism $\varphi:Q_1\to Q_2$. The other factors of $\overline{G}$ have trivial projection on these two coordinates. 
Projecting $L=\overline{G}L_{\overline{\a}}$ onto $Q_1\times Q_2$ gives $Q_1\times Q_2=\{(x,x^\varphi)\mid x\in Q_1\}\pi_{1,2}(L_{\overline{\a}})$. Taking an element of the form $(y,1)$ and comparing its two coordinates yields $Q_1=R_2^{\varphi^{-1}}R_1$. 
This is a factorization of the nonabelian simple group $Q_1$ by two isomorphic solvable proper subgroups, contrary to \cite[Lemma 5.2]{BPS2004}. Hence $T\not\cong Q$.

The argument used in Case 1 now shows that no $Q_i$ is contained in $\overline{G}$, so \cite[Theorem 1.3(3)]{DS2020} applies. The alternating possibility is excluded immediately because it gives $L_{\overline{\a}}\cong (\A_{n-1})^r$, whereas $L_{\overline{\a}}$ is solvable. In the strip case, each coordinate $Q_i$ belongs to the support of a unique factor $T_j$. Projection onto $Q_i$ yields $Q_i=\pi_i(T_j)R_i$, where $\pi_i(T_j)\cong T$ and $R_i<Q_i$ is solvable. Lemma \ref{lem:strip-pair-exclusion} (ii) excludes this possibility. Thus $\ell=r$ and, after relabeling, $T_i<Q_i$ for every $i$.

Let $c=|Q:T|$. Since $L=\overline{G}L_{\overline{\a}}$, the index $|L:\overline{G}|=c^r$ divides $|L_{\overline{\a}}|$. The point-stabilizer formula from \cite[Theorem 1.1(2)]{Lu2026} gives $c^r\mid p^tm_1m$. The action of $Q$ on the right cosets of $T$ is faithful because $\Core_Q(T)=1$, so $Q\leqslant \Sy_c$. If $c\leqslant5$, then the nonabelian simplicity of $Q$ forces $Q\cong\A_5$, which contains no proper nonabelian simple subgroup. Thus $c\geqslant 6$.

We now exclude the three possibilities in \cite[Theorem 1.1 (2)]{Lu2026}. Assume first that \cite[Theorem 1.1 (2)(i)]{Lu2026} holds and $(p,t)\neq (2,6)$. Let $z$ be the primitive prime divisor of $p^t-1$ specified there. Proposition \ref{prop:ppd} yields $z>t$, while \cite[Theorem 1.1 (2)(i)]{Lu2026} gives $z\mid r$. Hence $r>t$.

Assume that $p\mid c$. The factors $m_1$ and $m$ divide $p^e-1$, so neither is divisible by $p$. From $c^r\mid p^tm_1m$ we obtain $rv_p(c)\leqslant t$, contrary to $r>t$. Thus $p\nmid c$. It follows that $c^r\mid m_1m$. 
Moreover, $m_1m\leqslant (p^e-1)^2< p^{2e}$, and the proper divisibility $e\mid t$ gives $2e\leqslant t$. Hence $c^r<p^t$. The inequality $r>t$ now implies $c<p$. On the other hand, $p$ divides $|L_{\overline{\a}}|$, so $p$ divides $|Q|$ because $L_{\overline{\a}}\leqslant Q^r$. The embedding $Q\leqslant \Sy_c$ then forces $p\leqslant c$, a contradiction.

Consider next $(p,t)=(2,6)$. By \cite[Theorem 1.1 (2)(i)]{Lu2026}, either $7\mid r$ or $3\mid r$. We have $m_1m\leqslant 7^2$. If $7\mid r$, then $r\geqslant 7$, whereas $c\geqslant 6$, so $c^r\geqslant 6^7> 2^6\cdot 7^2\geqslant |L_{\overline{\a}}|$, impossible. Hence $3\mid r$. 
Every odd prime divisor of $m_1m$ occurs there with exponent at most 2. The divisibility $c^r\mid 2^6m_1m$ and $r\geqslant 3$ force $c$ to be a power of 2. Thus $c\geqslant 8$, and $rv_2(c)\geqslant 9$. This contradicts $v_2(|L_{\overline{\a}}|)=6$.

Suppose that \cite[Theorem 1.1 (2)(ii)]{Lu2026} holds. Then $(p,t)=(2,6)$ and $L$ is regular on the edges or arcs of $\Ga_N$. The valency of $\Ga_N$ is $2^6$, so $|L_{\overline{\a}}|$ is either $2^5$ or $2^6$. However, $r\geqslant 3$ and $c\geqslant 6$ give $c^r\geqslant 6^3>2^6$, contrary to $c^r\mid |L_{\overline{\a}}|$.

Next assume that \cite[Theorem 1.1 (2)(iii)]{Lu2026} holds. Then $t=2$ and $p$ is a Mersenne prime. If $p\mid c$, the divisibility $c^r\mid p^2m_1m$ gives $rv_p(c)\leqslant 2$, contrary to $r\geqslant 3$. Hence $p\nmid c$. 
The only proper divisor of $t$ is 1, so $m_1m\leqslant (p-1)^2$. Thus $c^r\mid m_1m$ and $c^r<(p-1)^2<p^2$. 
Then $c<p$ as $r\geqslant 3$. Yet $p\mid |L_{\overline{\a}}|$ implies $p\mid |Q|$, while $Q\leqslant \Sy_c$ forces $p\leqslant c$. This contradiction eliminates Case 2.
Hence Case 1 holds whenever $\overline{G}\ntrianglelefteq \overline{X}$, and the first assertion follows. In particular, $\ell=r\geqslant 2$.

Finally assume that $N\leqslant G$. Lemma \ref{lem:normal-quotient} shows that $\overline{G}$ is regular on $V(\Ga_N)$. Assume that $\overline{G}\unlhd \overline{X}$. The quasiprimitive group $\overline{X}$ has unique minimal normal subgroup $L$, so every nontrivial normal subgroup of $\overline{X}$ contains $L$. Thus $L\leqslant \overline{G}$. Together with $\overline{G}\leqslant L$, this gives $\overline{G}=L$, which is impossible because $\overline{G}$ is regular whereas $L_{\overline{\a}}\neq 1$. Hence $\overline{G}\ntrianglelefteq \overline{X}$.
The first assertion now gives $\ell=r\geqslant 2$. Assume that $\Ga_N$ is complete. Its $(\overline{X},2)$-arc-transitivity implies that $\overline{X}$ is 2-transitive on $V(\Ga_N)$. This is impossible for a product action group with $r\geqslant 2$, because a finite 2-transitive group is of affine or almost simple type. Thus $\Ga_N$ is not complete.
The regular subgroup $\overline{G}\cong T^\ell$ identifies $\Ga_N$ with a Cayley graph of $\overline{G}$. We have $\ell\geqslant 2$, $\overline{G}\ntrianglelefteq \overline{X}$ and $\Ga_N$ is not complete. Lemma \ref{lem:PA1} therefore applies to $(\Ga_N,\overline{G},\overline{X})$, with $\ell$ in place of $k$, and places $(Q,T,\ell)$ in Table \ref{tab:exact-pa}. This proves the second assertion and completes the proof. \qed

\begin{lemma}\label{lem:TW2}
If $\overline{X}$ is of twisted wreath type, then $\Ga_N$ is a Cayley graph of $\overline{G}=\soc(\overline{X})$.
\end{lemma}

\pf\, By Lemma \ref{lem:framework} (2)(e), we have $\overline{G}=\soc(\overline{X})$. The socle of a quasiprimitive group of twisted wreath type is regular, so $\overline{G}$ is regular and normal in $\overline{X}$. Thus $\Ga_N$ is an $\overline{X}$-normal Cayley graph of $\overline{G}$. \qed

\vspace{0.5em}

The preceding results determine the possible quasiprimitive types of $\overline{X}$ and the structure arising in each case. We now combine them to prove the non-quasiprimitive part of the main theorem.

\textbf{Proof of Theorem \ref{thm:non-quasiprimitive}}

\pf\, Let $\Ga=\Cay(G,S)$ be a connected $(X,2)$-arc-transitive Cayley graph, where $G=T_1\times\cdots\times T_k\cong T^k$, each $T_i\cong T$ is a finite nonabelian simple group, $k\geqslant 2$, and $G<X\leqslant\Aut(\Ga)$. Suppose that $X$ is not quasiprimitive on $V(\Ga)$, and let $N$ be a maximal intransitive normal subgroup of $X$. Fix $\a\in V(\Ga)$, let $\overline{\a}=\a^N$, $\overline{X}=X/N$ and $\overline{G}=GN/N$. By Lemma \ref{lem:normal-quotient}, the subgroup $N$ is semiregular on $V(\Ga)$, the graph $\Ga$ is an $N$-cover of $\Ga_N$, the graph $\Ga_N$ is nonbipartite and $(\overline{X},2)$-arc-transitive, and $\overline{X}$ is quasiprimitive on $V(\Ga_N)$. Moreover, $\overline{G}\cong T^\ell$ for some $1\leqslant\ell \leqslant k$, and $\overline{G}$ is regular on $V(\Ga_N)$ if and only if $N\leqslant G$.

By \cite[Theorem 1]{P1993}, the quasiprimitive group $\overline{X}$ is of affine, almost simple, simple diagonal, product action or twisted wreath type. Suppose first that $\overline{X}$ is of affine type. Lemma \ref{lem:affine2} gives $T\cong\PSL_2(7)$ and $\Ga_N\cong\K_{2^{3\ell}}$, which is exactly the affine alternative in the theorem.

Suppose next that $\overline{X}$ is of almost simple type, and let $L=\soc(\overline{X})$. By Lemma \ref{lem:AS2}, either $\Ga_N$ is complete, or $\ell=1$ and $\overline{G}\cong T\leqslant L$. In the latter case, $\overline{G}=L$, one of the cases listed in Tables \ref{tab:AS1}, \ref{tab:AS2}, \ref{tab:AS3}, \ref{tab:AS4} and \ref{tab:AS5} occurs, or $\Ga_N$ is $(L,2)$-arc-transitive. 
The simple diagonal type does not occur by Lemma \ref{lem:diagonal2}. 
Suppose now that $\overline{X}$ is of product action type, and let $L=\soc(\overline{X})=Q_1\times\cdots\times Q_r\cong Q^r$, where $r\geqslant 2$. If $\overline{G}\unlhd\overline{X}$, then the uniqueness of the minimal normal subgroup $L$ of the quasiprimitive group $\overline{X}$ gives $L\leqslant\overline{G}$. 
Proposition \ref{subgroup} gives $\overline{G}\leqslant L$, and so $\overline{G}=L$. Thus $T\cong Q$ and $\ell=r$, so $\soc(\overline{X})\cong Q^\ell$, as stated in the theorem. If $\overline{G}\ntrianglelefteq\overline{X}$, Lemma \ref{lem:PA2} gives $(Q,T,\ell)$ in Table \ref{tab:equal-prime-r}; in the proof of that lemma one has $\ell=r\geqslant 2$. 
If $N\leqslant G$, then $\overline{G}$ is regular on $V(\Ga_N)$ by Lemma \ref{lem:normal-quotient}, and the second assertion of Lemma \ref{lem:PA2} gives $(Q,T,\ell)$ in Table \ref{tab:exact-pa}. 
It remains to consider the twisted wreath type. By Lemma \ref{lem:TW2}, we have $\overline{G}=\soc(\overline{X})$, and $\Ga_N$ is an $\overline{X}$-normal Cayley graph of $\overline{G}$. 

The five quasiprimitive types have now been exhausted, and the proof is complete. \qed

\section{Proof of the main theorem}\label{sec6}

The quasiprimitive case was determined in Theorem \ref{thm:qprim}, while Theorem \ref{thm:non-quasiprimitive} completes the analysis when $X$ is not quasiprimitive on $V(\Ga)$. We now combine these two results to prove the main theorem.

\textbf{Proof of Theorem \ref{main}}

\pf\, Let $\Ga=\Cay(G,S)$ be a connected $(X,2)$-arc-transitive Cayley graph, where $G=T_1\times\cdots\times T_k\cong T^k$, each $T_i\cong T$ is a finite nonabelian simple group, $k\geqslant 2$, and $G<X\leqslant\Aut(\Ga)$.

Suppose first that $X$ is quasiprimitive on $V(\Ga)$. Theorem \ref{thm:qprim} gives the complete list of possibilities in this case. In particular, $\Ga$ is $X$-normal, $\Ga$ is complete, or $X$ is of product action type with $\soc(X)\cong Q^k$ and $(Q,T,k)$ as listed in Table \ref{tab:exact-pa}. These are exactly the alternatives in Theorem \ref{main} arising from the quasiprimitive action of $X$.

Suppose next that $X$ is not quasiprimitive on $V(\Ga)$. Let $N$ be a maximal intransitive normal subgroup of $X$, let $\overline{X}=X/N$, and let $\overline{G}=GN/N$. By Theorem \ref{thm:non-quasiprimitive}, the graph $\Ga$ is an $N$-cover of $\Ga_N$, the group $\overline{X}$ is quasiprimitive on $V(\Ga_N)$, and the affine, almost simple, product action and twisted wreath possibilities for $\overline{X}$ are precisely those listed in Theorem \ref{main}. The simple diagonal type does not occur.

Thus the cases in which $X$ is quasiprimitive and not quasiprimitive exhaust all possibilities for the action of $X$ on $V(\Ga)$, and all conclusions in Theorem \ref{main} follow.
\qed

\appendix
\section{Magma code}
\subsection{Magma verification 1}\label{app:M12}

The following Magma computation verifies that $\M_{12}$ has exactly two conjugacy classes of subgroups of order $144$, and that each representative has a unique conjugacy class of subgroups of order 12.

\begin{lstlisting}[basicstyle=\ttfamily\small,breaklines=true,columns=fullflexible]
M := PermutationGroup<12 | (1,2,3,4,5,6,7,8,9,10,11), (1,12,5,2,9,4,3,7)(6,10,11,8)>;
assert Order(M) eq 95040;
C := Subgroups(M : OrderEqual := 144);
assert #C eq 2;
for c in C do
    R := c`subgroup;
    assert #Subgroups(R : OrderEqual := 12) eq 1;
end for;
print "All assertions passed.";
\end{lstlisting}

\subsection{Magma verification 2}\label{app:A6-wreath}

The following Magma computation \cite{magma} verifies the subgroup intersections required in the proof and determines the subdegrees of the action of $\A_6$ on the cosets of $M=(\Sy_3\wr\Sy_2)\cap\A_6$.

\begin{lstlisting}[basicstyle=\ttfamily\small,breaklines=true,columns=fullflexible]
S := SymmetricGroup(6);
L := DerivedSubgroup(S);
W := sub<S | (1,2),(1,2,3),(4,5),(4,5,6),(1,4)(2,5)(3,6)>;
M := L meet W;
assert Order(M) eq 36;
assert #Subgroups(M : OrderEqual := 12) eq 0;
H18 := Subgroups(M : OrderEqual := 18);
assert #H18 eq 1;
A5 := { sub<L | (C`subgroup)^g > : C in Subgroups(L : OrderEqual := 60), g in L };
assert #A5 eq 12;
assert forall{ Order(T meet C`subgroup) eq 6 : T in A5, C in H18 };
phi, P := CosetAction(L,M);
assert Degree(P) eq 10;
assert Sort([#O : O in Orbits(Stabilizer(P,1))]) eq [1,9];
print "All assertions passed.";
\end{lstlisting}

\vskip 10pt
\noindent{\bf Declaration of competing interest}

The authors declare that they have no known competing financial interests or personal relationships that could have appeared to influence the work reported in this paper.

\end{document}